\documentclass[11pt]{article}
\usepackage[english]{babel}
\usepackage{a4wide}
\usepackage{amssymb, amsmath}
\usepackage{bm}
\usepackage{color}
\usepackage{enumitem}
\usepackage{amsthm}
\usepackage{wasysym}
\usepackage{stmaryrd}
\usepackage{stackrel}
\usepackage[capitalise]{cleveref}
\crefname{equation}{}{}
\usepackage{tikz-cd}
\usetikzlibrary{arrows}
\usepackage{url}

\newcommand{\pf}{\rightharpoonup}
\newcommand{\To}{\Rightarrow}

\newcommand{\denotes}{\!\downarrow}
\newcommand{\s}{\hspace{1pt}}

\newcommand{\asm}{\mathsf{Asm}}
\newcommand{\rt}{\mathsf{RT}}
\newcommand{\set}{\mathsf{Set}}
\renewcommand{\sf}[1]{\mathsf{#1}}
\newcommand{\kbar}{\overline{\sf{k}}}
\newcommand{\mono}{\hookrightarrow}
\newcommand{\epi}{\twoheadrightarrow}
\DeclareMathOperator{\dom}{dom}

\DeclareMathOperator{\downset}{\downarrow}
\DeclareMathOperator{\upset}{\uparrow}
\DeclareMathOperator{\id}{id}
\DeclareMathOperator{\er}{ER}
\DeclareMathOperator{\foe}{FOE}

\newtheorem{thm}{Theorem}[section]
\newtheorem{lem}[thm]{Lemma}
\newtheorem{cor}[thm]{Corollary}
\newtheorem{prop}[thm]{Proposition}
\theoremstyle{definition}
\newtheorem{defn}[thm]{Definition}
\newtheorem{rem}[thm]{Remark}
\newtheorem{ex}[thm]{Example}
\newtheorem{warn}[thm]{Warning}

\title{Third-order functionals on partial combinatory algebras}

\author{Jetze Zoethout\\
Department of Mathematics, Utrecht University}
\date{\today}

\begin{document}

\maketitle

\centerline{\textsc{abstract}}

\begin{footnotesize}
Computability relative to a partial function $f$ on the natural numbers can be formalized using the notion of an oracle for this function $f$. This can be generalized to arbitrary partial combinatory algebras, yielding a notion of `adjoining a partial function to a partial combinatory algebra $A$'. A similar construction is known for second-order functionals, but the third-order case is more difficult. In this paper, we prove several results for this third-order case. Given a third-order functional $\Phi$ on a partial combinatory algebra $A$, we show how to construct a partial combinatory algebra $A[\Phi]$ where $\Phi$ is `computable', and which has a `lax' factorization property (\cref{thm:type3} below). Moreover, we show that, on the level of first-order functions, the effect of making a third-order functional computable can be described as adding an oracle for a first-order function.
\end{footnotesize}

\section{Introduction}

Classical computability on the natural numbers can be extended with notions of computability involving functions. For example, we may add an \emph{oracle} for a non-computable partial function $f\colon \mathbb{N}\pf\mathbb{N}$ to our Turing machines, yielding a notion of computability relative to $f$. Another example is given by Kleene's S1-S9, which devises a system of computation with \emph{higher-order} functionals, whose inputs are also functionals, rather than just numbers.

This paper is concerned with computability on (higher-order) functions in the more general context of \emph{partial combinatory algebras} (PCAs), which can be viewed as abstract `models of computation'. Some work has previously been done on this subject. The paper \cite{Af} offers a notion of `oracle computability' for general PCAs, and the paper \cite{type2} generalizes this first paper to second-order functionals. The authors of \cite{type2} also mention that the third-order case seems to be far more difficult. The main result of this paper concerns this third-order case. Specifically, given a third-order functional $\Phi$ on a PCA $A$, we show how to construct a PCA $A[\Phi]$ where $\Phi$ is `computable', and which has a `lax' factorization property (\cref{thm:type3} below). Moreover, we show that, on the level of first-order functions, the effect of making a third-order functional computable can be described as adding an oracle for a first-order function.

As we said, this paper is concerned with computability on \emph{functions}. In the paper \cite{PCAfunctions}, it is shown that, for each PCA $A$, there is another PCA $\mathcal{B}A$ whose elements are partial \emph{functions} on $A$. The main strategy of this paper is to view higher-order functionals on $A$ as \emph{lower-order} functionals on $\mathcal{B}A$. In developing this strategy, the paper also serves as a unification of the material from \cite{Af} and \cite{type2} on the one hand, and \cite{PCAfunctions} on the other hand. However, there is also a serious obstacle for this strategy. It turns out that the elements of $\mathcal{B}A$ lead a kind of `double life'. On the one hand, they are elements of $\mathcal{B}A$, but on the other hand, for each $\alpha\in\mathcal{B}A$, there are (other) elements from $\mathcal{B}A$ that \emph{compute} $\alpha$ in a specific sense. These cannot always be translated into one another, and we will give an explicit example of this phenomenon (\cref{ex:counterexample}).

Let us briefly outline the paper. First of all, in \cref{sec:ROPCA}, we introduce partial combinatory algebras and the realizability toposes (denoted by $\rt$) that can be constructed out out of them. In this paper, PCAs will always be \emph{relative} and \emph{ordered}; the main reason for considering such a general notion of PCAs is that it allows us to construct a geometric surjection $\rt(\mathcal{B}A)\epi\rt(A)$. The relevant morphisms between PCAs are introduced in \cref{sec:pam}, which yields a preorder-enriched category of PCAs. \cref{sec:ROPCA} contains no new material, but \cref{sec:pam} treats a slight innovation with respect to the literature, namely the notion of a \emph{partial} applicative morphism, which is specific to \emph{relative} PCAs. In \cref{sec:order1}, we revisit the paper \cite{Af}, showing how to `freely adjoin' a partial function $f\colon A\pf A$ to a PCA $A$, yielding a new PCA $A[f]$. The construction is largely the same as in \cite{Af}, but we have made it suitable for the relative ordered case. Next, \cref{sec:BA} introduces the PCA $\mathcal{B}A$ of partial functions as in \cite{PCAfunctions}, but with one important deviation. One of the central views of this paper is that $\mathcal{B}A$ is best viewed as a relative, ordered PCA, even if $A$ itself is an `ordinary' PCA, i.e., lacking a notion of relativity and an order. In this section, we also show how, at the level of realizability toposes, $A[f]$ may be reconstructed using the construction $\mathcal{B}-$, slicing, and image toposes. This construction depends heavily on the fact that there is a geometric surjection $\rt(\mathcal{B}A)\epi\rt(A)$, and therefore, on our treatment of $\mathcal{B}A$ as a relative, ordered PCA. Then, in \cref{sec:type2}, we treat the second-order case, where we reinterpret the construction from \cite{type2} as a construction that actually takes place in $\mathcal{B}A$. We also give an explicit example of the `obstacle' mentioned above (\cref{ex:counterexample}). Finally, in \cref{sec:type3}, we present the aforementioned results on the third-order case.

\section{Relative ordered PCAs}\label{sec:ROPCA}

A \emph{partial combinatory algebra} is a non-empty set $A$ equipped with a \emph{partial} binary operation, called application. We think of the image of a pair $(a,b)$ under this application map as the result, if defined, of applying the algorithm (with code or G\"odel number) $a$ to the input $b$. In order to capture this computational intuiton, this application will need to satisfy a few requirements, to be specified below. In the current setting, we add two extra features to this application map. First of all, we equip our PCAs with a partial order. We think of $a'\leq a$ as saying that $a'$ gives more information than $a$, or that $a'$ is a refinement of $a$. Second, we specify a \emph{filter}, which is a subset of $A$ satisfying certain properties. We think of the elements of this filter as the `computable' elements, or as those (codes of) algorithms that can actually be carried out. Let us start by defining structures equipped with an application map and a partial order.

\begin{defn}\label{defn:PAP}
A \emph{partial applicative poset} (abbreviated PAP) is a triple $A=(A,\cdot,\leq)$ where $(A,\leq)$ is a poset and $\cdot$ is a partial binary map $A\times A\pf A$, called the \emph{application map}, such that the following axiom is satisfied:
\begin{itemize}
\item[(A)]	the application map has downwards closed domain and preserves the order, i.e., if $a'\leq a$, $b'\leq b$ and $a\cdot b$ is defined, then $a'\cdot b'$ is defined as well, and $a'\cdot b'\leq a\cdot b$.
\end{itemize}
The PAP $A$ is called \emph{total} if the application map is total, and \emph{discrete} if $\leq$ is the discrete order.
\end{defn}

Axiom (A) fits the informal intuition about the order on $A$: if $a'$ and $b'$ contain at least as much information as $a$ resp.\@ $b$, and $a\cdot b$ is already defined, then $a'\cdot b'$ should also be defined and contain at least as much information as $a\cdot b$.

Before we proceed to add the second extra feature mentioned above, let us agree on some notation and describe some basic constructions for PAPs. First of all, we will usually omit the dot for application, and just write $ab$ for $a\cdot b$. Since the application map is not required to be associative, and in fact usually \emph{will not} be associative, the bracketing of expressions is relevant. Here we adopt the convention, as is customary, that application associates to the left, meaning that $abc$ is an abbreviation of $(ab)c$.

Second, since the application map is partial, we will have to deal with expressions that may or may not be defined. If $e$ is a possibly undefined expression, then we write $e\denotes$ to indicate that $e$ is in fact defined. We take this to imply that all subexpressions of $e$ are defined as well. If $e$ and $e'$ are possibly undefined expressions, then we write $e'\preceq e$ to mean: if $e\denotes$, then $e'\denotes$ as well, and $e'\leq e$. Observe that axiom (A) above may now be rewritten as: if $a'\leq a$ and $b'\leq b$, then $a'b'\preceq ab$. On the other hand, we write $e'\leq e$ to indicate that $e'$ and $e$ are in fact defined, and $e'\leq e$. Observe that, in the discrete case, $e'\preceq e$ reduces to \emph{Kleene inequality}: if $e$ is defined, then $e'$ is also defined and denotes the same value.

Similarly, we write $e'\simeq e$ iff $e'\preceq e$ and $e\preceq e'$; in other words, $e\denotes$ precisely when $e'\denotes$, and in this case $e$ and $e'$ assume the same value. In other words, $\simeq$ is the familiar \emph{Kleene equality}. On the other hand, $e'=e$ expresses the stronger statement that $e'$ and $e$ are in fact defined, and equal to each other.

The following object will become increasingly important in the remainder of the paper.
\begin{defn}\label{defn:BA_as_poset}
Let $A$ be a PAP.
\begin{itemize}
\item[(i)]	We define $\mathcal{B}A$ as the set of all partial functions $\alpha\colon A\pf A$ such that $a\leq b$ implies $\alpha(a)\preceq \alpha(b)$ for all $a,b\in A$.
\item[(ii)]	For $\alpha,\beta\in \mathcal{B}A$, we say that $\alpha\leq\beta$ if $\alpha(a)\preceq \beta(a)$ for all $a\in A$.
\end{itemize}
\end{defn}
In other words, $\alpha\in \mathcal{B}A$ if and only if its domain is downwards closed, and $\alpha$ is order-preserving on its domain. The statement $\alpha\leq\beta$ means that the domain of $\alpha$ extends the domain of $\beta$, and $\alpha\leq \beta$ holds pointwise on the domain of $\beta$. Clearly, this makes $\mathcal{B}A$ into a poset; later,we shall see that it can be equipped with a PCA structure. Observe that, in the discrete case, $\mathcal{B}A$ is simply the set of all partial functions on $A$, and the order is the \emph{reverse} subfunction relation.

\begin{warn}\label{warn:Kleene}
Some authors (including myself in other papers) discussing the discrete case use the `opposite' convention for Kleene inequality, writing $e'\succeq e$ where we write $e'\preceq e$. The reason for doing so is that in this way, the corresponding order on partial functions is the actual subfunction relation, and not the reverse one. In the context of \emph{ordered} PCAs, however, the current convention is the right one to adopt. Indeed, in the case where expressions are defined, one should like $e'\preceq e$ to imply $e'\leq e$, and not $e'\geq e$. Moreover, the order defined on $\mathcal{B}A$ matches our intuition about orders. Indeed, $\alpha\leq \beta$ means that $\alpha$ provides more information than $\beta$, either by specifying more values than $\beta$, or by adding information to values already specified by $\beta$.
\end{warn}

\begin{ex}\label{ex:app=meet}
If $(A,\leq)$ is a poset with finite meets, then it can be made into a total PAP by setting $ab = a\wedge b$.
\end{ex}

\begin{ex}\label{ex:DA}
Let $A = (A,\cdot,\leq)$ be a PAP. We write $DA$ for the set of \emph{downwards closed} subsets of $A$, i.e., the set of all $\alpha\subseteq A$ satisfying: if $a'\leq a$ and $a\in \alpha$, then $a'\in \alpha$. Clearly, $DA$ is partially ordered by inclusion. We make $DA$ into a PAP by defining an application map as follows. If $\alpha,\beta\in DA$, then we say that $\alpha\beta\denotes$ if and only if $ab\denotes$ for all $a\in \alpha$ and $b\in\beta$. In this case, $\alpha\beta$ is defined as $\downset\{ab\mid a\in\alpha, b\in\beta\}$, i.e., the downwards closure of $\{ab\mid a\in\alpha, b\in\beta\}$.

Restricting $DA$ to the set $TA$ of \emph{non-empty} downwards closed subsets of $A$ also yields a PAP.
\end{ex}

The PAP $DA$ will play an important role in the sequel of the paper. The following notation will be convenient when working with $DA$.
\begin{defn}\label{defn:short}
For $a\in A$ and $\alpha\in DA$, we write
\[
a\alpha :\simeq \downset\{a\}\cdot\alpha \simeq \downset\{aa'\mid a'\in\alpha\},
\]
and a similar definition applies for expressions of the form $a\alpha\beta$, etc.
\end{defn}

We proceed to define filters.
\begin{defn}\label{defn:filter}
Let $A = (A,\cdot,\leq)$ be a PAP. A \emph{filter} of $A$ is an non-empty subset $F\subseteq A$ that is:
\begin{itemize}
\item[(i)]	closed under application, i.e., if $a,b\in F$ and $ab\denotes$, then also $ab\in F$;
\item[(ii)]	upwards closed, i.e., if $a\leq b$ and $a\in F$, then also $b\in F$.
\end{itemize}
\end{defn}

\begin{ex}\label{ex:filter=filter}
If $(A,\leq)$ is a poset with finite meets, then a filter on $(A,\wedge, \leq)$ is a filter in the usual order-theoretic sense.
\end{ex}

\begin{ex}\label{ex:filter_trans}
\begin{itemize}
\item[(i)]	If $A$ is a PAP, and $F$ is a filter on $A$, then $F$ can also be made into a PAP, by restricting both the application map and the order to $F$. This new PAP will be denoted by $(F,\cdot,\leq)$, or simply by $F$.
\item[(ii)]	If $A$ is a PAP, $F$ is a filter on $A$, and $G$ is a filter on the PAP $F$, then $G$ is also a filter on $A$.
\end{itemize}
\end{ex}

Since a filter is defined as a non-empty set with certain closure properties, we can consider the notion of a \emph{generated filter}.

\begin{defn}\label{defn:generate}
Let $A$ be a PAP and let $X$ be a non-empty subset of $A$. We define $\langle X\rangle$ as the smallest filter on $A$ extending $X$, and we call this the filter \emph{generated} by $X$. 
\end{defn}

In the case of filters on meet-semilattices, one can always generate a filter by first taking all finite meets, and then closing upwards. In the current case, a similar description is available. Befor we can formulate it, we need the notion of a \emph{term}, which will also be central to the definition of PCAs later in this section.

\begin{defn}\label{defn:term}
Let $A$ be a PAP. The set of \emph{terms} over $A$ is defined recursively as follows:
\begin{itemize}
\item[(i)]	We assume given a countably infinite set of disinct variables, and these are all terms.
\item[(ii)]	For every $a\in A$, we assume that we have a \emph{constant symbol} for $a$, and this is a term. The constant symbol for $a$ is simply denoted by $a$.
\item[(iii)]	If $t_0$ and $t_1$ are terms, then so is $(t_0\cdot t_1)$ (but we usually just write $(t_0t_1)$, and we omit brackets according to our convention).
\end{itemize}
\end{defn}
If $t=t(\vec{x})$ is a term whose variables are among the sequence $\vec{x}$, then we can assign an obvious, possibly undefined, interpretation $t(\vec{a})\in A$ to an input sequence $\vec{a}\in A$. In this way, every term $t(\vec{x})$ yields a partial function $\lambda\vec{a}.t(\vec{a})\colon A^n\pf A$, where $n$ is the length of the sequence $\vec{x}$.

We have the following alternative descriptions of generated filters; the proof is easy and omitted.
\begin{lem}\label{lem:generate}
Let $A$ be a PAP and $X\subseteq A$ be non-empty. Then
\[
\langle X\rangle = \upset\{t(\vec{a})\mid t(\vec{x})\mbox{ \textup{a constant-free term}}, \vec{a}\in X\mbox{ \textup{and} }t(\vec{a})\denotes\},
\]
where $\upset$ stands for taking the upwards closure.
\end{lem}

As promised, we will consider partial applicative preorders equipped with a filter.
\begin{defn}
A \emph{partial applicative structure} (abbreviated PAS) is a quadruple $A = (A,A^\#, \cdot, \leq)$, where $(A,\cdot,\leq)$ is a PAP, and $A^\#$ is a filter on $(A,\cdot,\leq)$. So explicitly, $A^\#$ is a non-empty subset of $A$ satisfying the following axioms:
\begin{itemize}
\item[(B)]	$A^\#$ is closed under application;
\item[(C)]	$A^\#$ is upwards closed.
\end{itemize}
The PAS $A$ is called \emph{absolute} if $A^\# = A$. Moreover, a \emph{filter} on $A$ is a filter $F$ on $(A,\cdot,\leq)$ such that $A^\#\subseteq F$.
\end{defn}

\begin{ex}\label{ex:filter=PAS}
Let $A$ be a PAS and let $F$ be a filter on $A$. Then $(F,A^\#,\cdot,\leq)$ is also a PAS. When no confusion can arise, we will denote this PAS simply  by $F$. Of course, $(A,F,\cdot,\leq)$ is also a PAS.
\end{ex}

\begin{ex}\label{ex:DA=PAS}
If $F$ is a filter on the PAP $A$, then $\{\alpha\in DA\mid \alpha\cap F\neq \emptyset\}$ is a filter on the PAP $DA$. In particular, if $A$ is a PAS, then we can make $DA$ into a PAS as well by setting $(DA)^\# = \{\alpha\in DA\mid \alpha\cap A^\#\neq\emptyset\}$. Similar remarks hold for $TA$.
\end{ex}

Thus far, all the structure we have introduced is of an algebraic nature, and does not yet express a notion of computability. Now let us finally introduce the `computational' component of PCAs.

\begin{defn}\label{defn:PCA}
A \emph{partial combinatory algebra} (abbreviated PCA) is a PAS $A$ for which there exist $\sf{k},\sf{s}\in A^\#$ such that:
\begin{itemize}
\item[(D)]	$\sf{k}ab \leq a$;
\item[(E)]	$\sf{s}ab\denotes$;
\item[(F)]	$\sf{s}abc\preceq ac(bc)$,
\end{itemize}
for all $a,b,c\in A$. 
\end{defn}
The elements $\sf{k}$ and $\sf{s}$ are usually called \emph{combinators}. Using these combinators, every computation using the application map can be represented by a computable element (i.e., algorithm) from $A$ itself. In order to make this statement precise, we use the terms introduced in \cref{defn:term}. As we mentioned, every term defines a partial function $A^n\pf A$. The key fact about PCAs is that such partial functions are (laxly) computable using an element from $A$ itself.
\begin{prop}[Combinatory completeness]\label{prop:comb_comp}
Let $A$ be a PCA. There exists a map that assigns to each term\footnote{Strictly speaking: a term along with an ordered sequence of distinct variables containing the variables from the term. A more formal treatment could be given using terms-in-context, but we will not take this trouble here.} $t = t(\vec{x},y)$ with at least one variable, an element $\lambda^\ast \vec{x},y.t \in A$, satisfying:
\begin{itemize}
\item[\textup{(}i\textup{)}]	$(\lambda^\ast \vec{x},y.t)\vec{a}\denotes$ (where $\vec{a}$ has the same length as $\vec{x}$);
\item[\textup{(}ii\textup{)}]	$(\lambda^\ast \vec{x},y.t)\vec{a}b\preceq t(\vec{a},b)$;
\item[\textup{(}iii\textup{)}]	if all the constants occurring in $t$ are from $A^\#$, then $\lambda^\ast\vec{x},y.t\in A^\#$ as well.
\end{itemize}
\end{prop}
\begin{proof}
Define the element $\sf{i}\in A^\#$ as $\sf{skk}$. We will give a slightly more general construction than required for the proposition. For a variable $u$ and a term $s$, we define a new term $\lambda^\ast u.s$ with the following properties:
\begin{itemize}
\item	the free variables of $\lambda^\ast u.s$ are those of $s$ minus $u$;
\item	if $\vec{v}$ are the free variables of $\lambda^\ast u.s$, then the substitution instance $(\lambda^\ast u.s)[\vec{b}/\vec{v}]$ is defined for all $\vec{b}\in A$;
\item	moreover, if $a\in A$, then $(\lambda^\ast u.s)[\vec{b}/\vec{v}]\cdot a\preceq s[\vec{b}/\vec{v},a/u]$;
\item	if all the constants occuring in $s$ are from $A^\#$, then the same holds for $\lambda^\ast u.s$.
\end{itemize}
We define this new term resursively:
\begin{itemize}
\item	If $s$ is a constant or a variable distinct from $u$, then $\lambda^\ast u.s$ is $\sf{k}s$.
\item	If $s$ is the variable $u$, then $\lambda^\ast u.s$ is $\sf{i}$.
\item	If $s$ is $s_0s_1$, then $\lambda^\ast u.s$ is $\sf{s}(\lambda^\ast u.s_0)(\lambda^\ast u.s_1)$.
\end{itemize}
We leave the verification of the stated properties to the reader.

Now, if $\vec{x} = x_0, \ldots, x_{n-1}$, then we define $\lambda^\ast\vec{x},y.t$ as (the interpretation of) the closed term
\[
\lambda^\ast x_0.(\cdots(\lambda^\ast x_{n-1}.(\lambda^\ast y.t))\cdots).
\]
The verification of the properties (i), (ii) and (iii) is also left to the reader. Details may also be found in Chapter 1 of \cite{jaap} (which treats the discrete, absolute case, but this can easily be generalized to our case).
\end{proof}

Some useful combinators besides $\sf{k}$ and $\sf{s}$ are $\sf{i}=\sf{skk}$ defined above, $\kbar = \sf{ki}$, $\sf{p} = \lambda^\ast xyz.zxy$, $\mathsf{p}_0 = \lambda^\ast x. x\sf{k}$ and $\sf{p}_1 = \lambda^\ast x. x\kbar$. Observe that these combinators all belong to $A^\#$ and satisfy:
\[
\sf{i}a\leq a,\quad \kbar ab\leq b, \quad \sf{p}_0(\sf{p}ab)\leq a \quad\mbox{and}\quad \sf{p}_1(\sf{p}ab)\leq b.
\]
In particular, $\sf{p}ab$ is always defined, and we think of this element as (coding) the \emph{pair} $(a,b)$. Accordingly, $\sf{p}$ is called the pairing combinator, and $\sf{p}_0$ and $\sf{p}_1$ are known as the unpairing combinators.

Moreover, we can construct \emph{booleans}, i.e. elements $\top, \bot\in A^\#$ for which there is a \emph{case operator} $\sf{C}\in A^\#$ satisfying $\sf{C}\top ab\leq a$ and $\sf{C}\bot ab\leq b$. Indeed, we may simply take $\top = \sf{k}$, $\bot = \kbar$ and $\sf{C} = \sf{i}$. When we are dealing with expressions that are possibly undefined, we need to be a bit more careful. Suppose we have terms $t_0(\vec{x})$, $t_1(\vec{x})$ and $t_2(\vec{x})$, and define the new term $t:=\sf{C}t_0t_1t_2$, whose free variables are also among $\vec{x}$. Then this term does not behave as one would expect at first glance. In particular, if $t_0(\vec{a})\leq\top$ and $t_1(\vec{a})\denotes$, then it does \emph{not} follow that $t(\vec{a})$ is defined. Indeed, it may happen that $t_2(\vec{a})$ fails to be defined and, since $t_2$ is a subterm of $t$, this prevents $t(\vec{a})$ from being defined. We clearly do not want this, since we are not interested in the value (if any) of $t_2(\vec{a})$ when $t_0(\vec{a})\leq\top$. Therefore, we introduce a \emph{strong case distinction} (we take this terminology from \cite{HOC}, Section 3.3.3.). If $t_0(\vec{x})$, $t_1(\vec{x})$ and $t_2(\vec{x})$ are terms, then we define a new term $t'(\vec{x})$ as:
\[
\sf{C}t_0(\lambda^\ast y. t_1)(\lambda^\ast y. t_2)\sf{i}.
\]
where $y$ is not among the $\vec{x}$. One can easily check that this term has the following property: if $t_0(\vec{a})\leq \top$, then $t'(\vec{a})\preceq t_1(\vec{a})$, whereas if $t_0(\vec{a})\leq \bot$, then $t'(\vec{a})\preceq t_2(\vec{a})$. We will denote the term $t'$ above by $\sf{if}\ t_0\ \sf{then}\ t_1\ \sf{else}\ t_2$. Observe that, if all parameters from $t_0$, $t_1$ and $t_2$ are in $A^\#$, then the same holds for $\sf{if}\ t_0\ \sf{then}\ t_1\ \sf{else}\ t_2$.

As in ordinary recursion theory, we have \emph{fixpoint operators}. Using the terminology from \cite{HOC}, Section 3.3.5, we have a fixpoint operator $\sf{y}\in A^\#$ and a guarded fixpoint operator $\sf{z}\in A^\#$ satisfying: $\sf{y}a\preceq a(\sf{y}a)$, $\sf{z}a\denotes$ and $\sf{z}ab\preceq a(\sf{z}a)b$. These may be constructed as $\sf{y} = uu$ where $u=\lambda^\ast xy.y(xxy)$ and $\sf{z} = vv$ where $v=\lambda^\ast xyz.y(xxy)z$. The fixpoint operator $\sf{y}$ is generally only useful in total PCAs, since $\sf{y}a\preceq a(\sf{y}a)$ will always be true if $\sf{y}a$ is not defined. The guarded fixpoint operator $\sf{z}$ has the property that $\sf{z}a$ is always defined, and can be used to create self-referential definitions. Explicitly, if $t(x,y)$ is a term, then setting $a:=\lambda^\ast xy. t$ yields an element $T := \sf{z}a$ with the property that $Tb \preceq aTb \preceq t(T,b)$ for all $b\in A$. Moreover, if all the parameters from $t$ are in $A^\#$, then $T\in A^\#$ as well. Obviously, this construction can be generalized to more variables, either by adjusting the definition of $\sf{z}$ or by using the pairing combinators.

All this justifies the view of PCAs as generalizing computability on the natural numbers. In fact, we can code the natural numbers in a PCA $A$, by setting recursively $\overline{0} = \sf{i}$ and $\overline{n+1} = \sf{p}\bot \overline{n}$; observe that all the $\overline{n}$ are in $A^\#$. Usually, we will omit the bar and simply write $n\in A^\#$, where we really mean its representative $\overline{n}$. Now it is easily checked that the elements $\sf{zero} = \sf{p}_0$, $\sf{suc} = \lambda^\ast x. \sf{p}\bot x$ and $\sf{pred} = \lambda^\ast x. \sf{p}_0x\sf{i}(\sf{p}_1x)$ from $A^\#$ satisfy: $\sf{zero}\cdot 0 \leq \top$, $\sf{zero}\cdot (n+1)\leq \bot$, $\sf{suc}\cdot n\leq n+1$, $\sf{pred}\cdot 0 \leq 0$ and $\sf{pred}\cdot (n+1)\leq n$. Moreover, using the guarded fixpoint operator, we may construct a recursor $\sf{rec}\in A^\#$ such that
\[
\sf{rec}  a b 0\leq a\quad\mbox{and}\quad\sf{rec} a b(n+1) \preceq bn(\sf{rec} a b n)\mbox{ for all }n\in\mathbb{N}\mbox{ and }a,b\in A.
\]
Since we have a `pairing' function given by $\sf{p}\in A^\#$, we can also code longer tuples in $A$. More precisely, we can define total functions $j^n\colon A^n\to A$ for $n\geq 0$ by:
\begin{itemize}
\item	$j^0() = \sf{i}$;
\item	$j^{n+1}(a_0, \ldots, a_n) = \sf{p}a_0\cdot j^n(a_1, \ldots, a_n)$.
\end{itemize}
Using these functions, we can devise a \emph{coding of finite sequences in $A$}. If $a_0, \ldots, a_{n-1}$ is a sequence, then we define its code by:
\[
[a_0, \ldots, a_{n-1}]:= \sf{p}n\cdot j^n(a_0, \ldots, a_{n-1}).
\]
Observe that $[a_0, \ldots, a_{n-1}]$ is built using the $a_i$, combinators from $A^\#$ and application. In particular, if all the $a_i$ are from $A^\#$, then so is the code $[a_0, \ldots, a_{n-1}]$.

Using the combinators above, one can mimick the standard recursion theoretic arguments to show that all elementary operations on sequences are computable in terms of their codes. The following definition introduces a few such computations that we will need in the sequel.
\begin{defn}\label{defn:rij}
If $A$ is a PCA, then $\sf{lh}, \sf{read},  \sf{fst}, \sf{concat}, \sf{unit}, \sf{ext}\in A^\#$ are combinators that satisfy:
\begin{itemize}
\item	$\sf{lh}\cdot[a_0, \ldots, a_{n-1}] \leq n$;
\item	$\sf{read}\cdot [a_0, \ldots, a_{n-1}]\cdot i\leq a_i$ if $i<n$;
\item	$\sf{fst}\cdot [a_0, \ldots, a_n] \leq a_0$;
\item	$\sf{concat}\cdot [a_0, \ldots, a_{n-1}]\cdot [b_0, \ldots, b_{m-1}] \leq [a_0, \ldots, a_{n-1}, b_0, \ldots, b_{m-1}]$;
\item	$\sf{unit}\cdot a\leq [a]$;
\item	$\sf{ext}\cdot[a_0, \ldots, a_{n-1}]\cdot a'\leq [a_0, \ldots, a_{n-1}, a']$.
\end{itemize} 
\end{defn}

\begin{rem}\label{rem:ks}
Of course, the combinators constructed up to this point are far from unique. But all of them may be constructed using only the elements $\sf{k}$ and $\sf{s}$. When working with a PCA, we will assume that we have made an explicit choice for $\sf{k}$ and $\sf{s}$, and as a result, a choice for all the combinators mentioned above.
\end{rem}

\begin{ex}\label{ex:filter=PCA}
If $A$ is a PCA and $F$ is a filter on $A$, then $(A,F,\cdot,\leq)$ and $(F,A^\#,\cdot,\leq)$ are also PCAs, as can be seen by taking the same combinators $\sf{k}$ and $\sf{s}$.

The following instance of this example will be relevant in the coming sections. If $r$ is an element of $A$, then we define $F_r = \langle A^\#\cup\{r\}\rangle$, i.e., $F_r$ is the least filter on $A$ containing $r$. We denote the PCA $(A,F_r,\cdot,\leq)$ by $A[r]$.
\end{ex}

\begin{ex}\label{ex:DA=PCA}
If $A$ is a PCA, then so are $DA$ and $TA$. In both cases, a suitable choice of combinators is $\downset\{\sf{k}\}, \downset\{\sf{s}\}$.
\end{ex}

\begin{ex}\label{ex:K1}
The prototypical example of a (discrete, absolute) PCA is \emph{Kleene's first model} $\mathcal{K}_1$. Its underlying set is $\mathbb{N}$, and $m\cdot n$ is the result, if any, of applying the $m^\text{th}$ partial recursive function to $n$.
\end{ex}

\begin{defn}\label{defn:semitrivial}
A PCA $A$ is called \emph{semitrivial} if $\top,\bot\in A$ have a common lower bound.
\end{defn}

This notion is introduced for the following reason: many constructions in this paper use case distinctions inside a PCA $A$. Usually, such constructions do not work in a semitrivial PCA, since we cannot distinguish $\top$ and $\bot$. Observe that, if $u$ is a common lower bound of $\top$ and $\bot$, then $uab$ is a common lower bound of $a$ and $b$, for any $a,b\in A$. So in a semitrivial PCA, every two elements have a common lower bound. On the other hand, if $A$ is not semitrivial, then the numerals for any two distinct $m,n\in \mathbb{N}$ do not have a common lower bound. In particular, every non-semitrivial PCA is infinite.

In the next section, we will introduce the relevant morphisms between PCAs. Before we move to this section, we briefly describe two important categorical constructions on PCAs.

\begin{defn}\label{defn:asm}
Let $A$ be a PCA.
\begin{itemize}
\item[(i)]	An \emph{assembly} over $A$ is a pair $X = (|X|,E_X)$, where $|X|$ is a set and $E_X$ is a function $|X|\to TA$, i.e., $E_X(x)$ is a non-empty downwards closed subset of $A$, for all $x\in |X|$.
\item[(ii)]	A \emph{morphism of assemblies} $X\to Y$ is a function $f\colon |X|\to |Y|$ for which there exists a $t\in A^\#$ such that: for all $x\in |X|$, the set $t\cdot E_X(x)$ (as in \cref{defn:short}) is defined and a subset of $E_Y(f(x))$. Such a $t$ is called a \emph{tracker} for $f$.
\end{itemize}
\end{defn}

\begin{prop}
Assemblies over a PCA $A$ and morphisms between them form a category $\asm(A)$, and this is a quasitopos.
\end{prop}
The proof is a straightforward generalization of the material in Section 1.5 from \cite{jaap}, and is omitted. There is an obvious forgetful functor $\Gamma\colon\asm(A)\to\set$ sending $X$ to $|X|$ and which is the identity on arrows. In the other direction, there is a functor $\nabla\colon \set\to\asm(A)$ given by $\nabla Y = (Y,\lambda y. A)$ and $\nabla f = f$. The functors $\Gamma$ and $\nabla$ are both regular, and we have $\Gamma\dashv \nabla$ with $\Gamma\nabla\cong\id_\set$.

Even though $\asm(A)$, being a quasitopos, enjoys nice properties, it is not an exact category. We can make it exact by taking the ex/reg completion, which turns out to be an elementary topos.

\begin{defn}\label{defn:rt}
Let $A$ be a PCA. The \emph{realizability topos} $\rt(A)$ is defined as $\asm(A)_\text{ex/reg}$.
\end{defn}

Since $\Gamma\colon\asm(A)\to\set$ is regular and $\set$ is exact, this functor may be lifted to a functor $\rt(A)\to \set$, which will also be denoted by $\Gamma$. In the other direction, we have the composition $\set\stackrel{\nabla}{\longrightarrow} \asm(A)\mono \rt(A)$, which will also be denoted by $\nabla$. This yields a geometric inclusion $\Gamma\dashv\nabla\colon \set\to \rt(A)$, which is equivalent to the inclusion of $\neg\neg$-sheaves of $\rt(A)$.


\section{Partial applicative morphisms}\label{sec:pam}

In this section, we introduce morphisms between PCAs. Usually, a morphism from a PCA $A$ to a PCA $B$ is a function that assigns to each $a\in A$ a non-empty subset of $B$. In the ordered setting, this needs to be amended to: a non-empty \emph{downset} of $B$. The non-emptiness condition is needed to make sure that every computation from $A$ can be `transferred' along the morphism to $B$. However, in the relative setting, computations in $A$ are given by elements from $A^\#$, rather than $A$. So one really needs to require the following: for each $a\in A^\#$, the associated downset of $B$ contains an element of $B^\#$. For $a$ outside $A^\#$, the non-emptiness conditions can be omitted, leading to the notion of a \emph{partial applicative morphism}.

\begin{defn}\label{defn:pam}
Let $A$ and $B$ be PCAs.
\begin{itemize}
\item[(i)]	A \emph{partial applicative morphism} $A\pf B$ is a function $f\colon A\to DB$ satisfying the following three requirements:
\begin{enumerate}
\item	$f(a)\cap B^\#\neq\emptyset$ for all $a\in A^\#$.
\item	There exists a $t\in B^\#$ such that: for all $a,a'\in A$, if $aa'$ is defined, then $t\cdot f(a)\cdot f(a')$ is defined as well, and a subset of $f(aa')$. Such a $t$ is called a \emph{tracker} for $f$.
\item	There exists a $u\in B^\#$ such that: $u\cdot f(a)$ is defined for all $a\in A$, and $u\cdot f(a)\subseteq f(a')$ whenever $a\leq a'$. We say that $f$ preserves the order up to $u$.
\end{enumerate}
We say that $f$ is \emph{total} if $f(a)\neq\emptyset$ for all $a\in A$, and in this case, we write $f\colon A\to B$.
\item[(ii)]	If $f,f'\colon A\to DB$ are functions, then we say that $f\leq f'$ if there exists an $s\in B^\#$ such that: for all $a\in A$, the set $s\cdot f(a)$ is defined and a subset of $f'(a)$. Such an $s$ is said to \emph{realize} the inequality $f\leq f'$. Moreover, we write $f\simeq f'$ if both $f\leq f'$ and $f'\leq f$.
\end{itemize}
\end{defn}

\begin{prop}\label{prop:pPCA}
PCAs, partial applicative morphisms and inequalities between them form a preorder-enriched category $\sf{pPCA}$. Restricting the 1-cells to \emph{total} applicative morphisms yields another preorder-enriched category $\sf{PCA}$.
\end{prop}
\begin{proof}
The identity on a PCA $A$ is given by $\id_A(a) = \downset\{a\}$. Moreover, if $A\stackrel{f}{\pf}B\stackrel{g}{\pf} C$, then their composition $gf$ is defined by $gf(a) = \bigcup_{b\in f(a)} g(b)$. The proof that this yields a preorder-enriched category proceeds as in Section 1.5 of \cite{jaap}, with appropriate adjustments. We do mention that $\sf{pPCA}$ is not a \emph{strict} preorder-enriched category, since $\id_Af\simeq f$ in general holds only up to isomorphism (see also \cref{lem:on_the_nose} below). The other equations for a category do hold strictly.

For the second statement, observe that identities are total, and that total morphisms are closed under composition.
\end{proof}

The reason why we call our applicative morphisms `partial' is that we view the fact that $f(a)$ is non-empty as evidence that $f(a)$ is `actually defined'. Moreover, an element of $f(a)\cap B^\#$ counts as \emph{effective} evidence that $f(a)$ is defined. This motivates the following definition.

\begin{defn}\label{defn:dom}
Let $f\colon A\pf B$ be a partial applicative morphism. Then we write
\[
\dom f = \{a\in A\mid f(a)\neq\emptyset\}\quad\mbox{and}\quad \dom^\# f = \{a\in A\mid f(a)\cap B^\#\neq\emptyset\}.
\]
\end{defn}
We see that $f$ is total precisely when $\dom f = A$. Since $f\leq f'$ implies that $\dom f\subseteq \dom f'$, this also means that $\sf{PCA}(A,B)$ is upwards closed inside $\sf{pPCA}(A,B)$. Moreover, we have the following result.

\begin{lem}\label{lem:dom=filter}
Let $f\colon A\pf B$ be a partial applicative morphisms. Then $\dom f$ and $\dom^\# f$ are filters of $A$.
\end{lem}
\begin{proof}
By property 1 of $f$, we have $A^\#\subseteq \dom^\# f\subseteq \dom f$. If $t\in B^\#$ tracks $f$ and $a,a'\in A$ are such that $aa'\denotes$, then $t\cdot f(a)\cdot f(a')\subseteq f(aa')$. So, if $b\in f(a)$ and $b'\in f(a')$, then $tbb'\in f(aa')$, which shows that $\dom f$ is closed under application. Similarly, if $b\in f(a)\cap B^\#$ and $b'\in f(a')\cap B^\#$, then $tbb'\in f(aa')\cap B^\#$, so $\dom^\# f$ is closed under application as well. The proof that $\dom f$ and $\dom^\# f$ are upwards closed proceeds similarly, using property 3 of $f$.
\end{proof}

\begin{ex}\label{ex:iota_r}
Let $A$ be a PCA, let $r\in A$ and recall the PCA $A[r]$ from \cref{ex:filter=PCA}. Then there exists a total applicative morphism $\iota_r\colon A\to A[r]$ that `acts as the identity', i.e., $\iota_r(a) = \downset\{a\}$. This morphism has the following universal property: if $f\colon A\pf B$, then $f$ factors, up to isomorphism, through $\iota_r$ if and only if $r\in \dom^\# f$. Indeed, $f$ is also a partial applicative morphism $A[r]\pf B$ if and only if $F_r\subseteq \dom^\# f$; and since $\dom^\# f$ is a filter on $A$, this is true if and only if $r\in \dom^\# f$.
\end{ex}

When working with partial applicative morphisms, the following lemma (\cite{hofstrajaap}, Lemma 3.3) is often useful. It says that every partial applicative morphism is isomorphic to one that preserves the order `on the nose'.
\begin{lem}\label{lem:on_the_nose}
Every partial applicative morphism $f\colon A\pf B$ is isomorphic to an order-preserving partial applicative morphism $f'$, meaning that $a\leq a'$ implies $f'(a)\subseteq f'(a')$.
\end{lem}
\begin{proof}
Define $f'$ as $f\id_A$, i.e., $f'(a) = \bigcup_{a'\leq a} f(a')$ for $a\in A$. \cref{prop:pPCA} tells us that $f'$ is a partial applicative morphism which is isomorphic to $f$. Moreover, it is immediately clear that $f'$ preserves the order on the nose.
\end{proof}

\begin{rem}
Ordered PCAs were introduced in the paper \cite{hofstrajaap}. Here, the authors consider a category whose objects are ordered (but absolute) PCAs, but whose arrows are \emph{functions} $A\to B$. They show that the assignment $A\mapsto TA$ is part of a monad structure on this category, whose Kleisli category is (the absolute version of) $\sf{PCA}$. For $\sf{pPCA}$, a similar treatment can be given using the construction $A\mapsto DA$ instead. In the absolute setting, the PCA $DA$ is not very interesting. Indeed, because $DA$ has a least element $\emptyset$, it is equivalent to the one-element PCA. In the relative setting, on the other hand, this is no longer true, due to the fact that the least element $\emptyset$ of $DA$ is never in $(DA)^\#$.
\end{rem}

We now proceed to extend the constructions $\asm$ and $\rt$ to partial applicative morphisms.
\begin{defn}\label{defn:Asm(f)}
Let $f\colon A\pf B$ be a partial applicative morphism. We define the functor $\asm(f)\colon \asm(A)\to \asm(B)$ by:
\begin{itemize}
\item	$|\asm(f)(X)| = \{x\in|X|\mid E_X(x)\cap\dom f\neq \emptyset\}$;
\item	$E_{\asm(f)(X)}(x) = \bigcup_{a\in E_X(x)} f(a)$ for $x\in |\asm(f)(X)|$;
\item	$\asm(f)(g)$ is the restriction of $g$ to $|\asm(f)(X)|$,
\end{itemize}
for assemblies $X\in\asm(A)$ and arrows $g\colon X\to Y$ of $\asm(A)$.
\end{defn}

Since we are working with multiple assemblies simultaneously, we will sometimes write $\Gamma_A\dashv \nabla_A$ for the adjunction between $\set$ and $\asm(A)$ or $\rt(A)$. A functor $F\colon \asm(A)\to \asm(B)$ is called a $\Gamma$-functor if $\Gamma_B F\cong \Gamma_A$, and similarly, a $\nabla$-functor if $F\nabla_A\cong\nabla_B$. For functors $\rt(A)\to \rt(B)$, we adopt a similar definition. Now we can state some elementary properties of the functor $\asm(f)$.

\begin{prop}
For every partial applicative morphism $f$, the functor $\asm(f)$ is a regular $\nabla$-functor. Moreover, it is a $\Gamma$-functor precisely when $f$ is total.
\end{prop}

By the universal property of the ex/reg completion, there is a unique regular functor $\rt(A)\to\rt(B)$ whose restriction to assemblies is $\asm(f)$. We will denote this functor by $\rt(f)$; this is also a regular $\nabla$-functor. Moreover, $\rt(f)$ is a $\Gamma$-functor iff $\asm(f)$ is a $\Gamma$-functor, iff $f$ is total.

\begin{rem}
Longley has shown (\cite{longleyphd}, Sections 2.2 and 2.3), for the discrete, absolute case, that a left exact functor $F\colon \asm(A)\to \asm(B)$ is a $\nabla$-functor if and only if it is a $\Gamma$-functor. Morover, the regular $\nabla$-functors are (up to equivalence) precisely the functors of the form $\asm(f)$. This can be generalized to the ordered case, but the relative case is slightly different. Indeed, while $\Gamma$-functors are still always $\nabla$-functors, the converse is not true. In the relative case, the regular $\nabla$-functors $\asm(A)\to\asm(B)$ correspond to the partial applicative morphisms $A\pf B$, while the regular $\Gamma$-functors $\asm(A)\to\asm(B)$ correspond to \emph{total} applicative morphisms $A\to B$.
\end{rem}

In order to complete the picture, we remark that, if $f'\leq f$, we get a natural transformation $\mu\colon \asm(f)\To\asm(f')$, where $\mu_X$ is the inclusion $|\asm(f)(X)|\subseteq |\asm(f')(X)|$ for assemblies $X$. By the universal property of the ex/reg completion, we also get a natural transformation $\rt(f)\To\rt(f')$. This makes $\asm$ and $\rt$ into pseudofunctors from $\sf{pPCA}$ to the 2-category of categories.

Now we discuss some special properties that a partial applicative morphism may have.
\begin{defn}\label{defn:pam_properties}
Let $f\colon A\pf B$ be a partial applicative morphism. Then $f$ is called:
\begin{itemize}
\item[(i)]	\emph{decidable} if there is a $d\in B^\#$ (called a \emph{decider} for $f$) such that $d\cdot f(\top)\subseteq \downset\{\top\}$ and $d\cdot f(\bot)\subseteq \downset\{\bot\}$;
\item[(ii)]	\emph{computationally dense} (abbreviated c.d.) is there is an $n\in B^\#$ such that:
\begin{align}\label{eq:cd}\tag{cd}
\forall s\in B^\#\s \exists r\in A^\#\s (n\cdot f(r)\subseteq \downset\{s\});
\end{align}
\item[(iii)]	\emph{single-valued} if $f(a)$ is a principal downset of $B$ for each $a\in \dom f$;
\item[(iv)]		\emph{projective} if $f\simeq f'$ for some single-valued $f'\colon A\pf B$.
\end{itemize}
\end{defn}
We list some elementary results on decidability and computational density. The proofs are left to the reader.
\begin{prop}\label{prop:properties_properties}
Let $A\stackrel{f}{\pf} B\stackrel{g}{\pf} C$ be partial applicative morphisms.
\begin{itemize}
\item[\textup{(}i\textup{)}]	if $f$ is c.d., then $f$ is decidable;
\item[\textup{(}ii\textup{)}]	if $f'\colon A\pf B$ satisfies $f'\leq f$ and $f$ is c.d.\@ (resp.\@ decidable), then $f'$ is c.d.\@ (resp.\@ decidable) as well;
\item[\textup{(}iii\textup{)}]	if $f$ and $g$ are c.d.\@ (resp.\@ decidable), then $gf$ is c.d.\@ (resp.\@ decidable) as well;
\item[\textup{(}iv\textup{)}]	if $gf$ is c.d.\@ (resp.\@ decidable), then $g$ is c.d.\@ (resp.\@ decidable) as well.
\end{itemize}
In particular, left adjoints are c.d.\@ and decidable.
\end{prop}

Computational density was first introduced in the paper \cite{hofstrajaap}. The definition we gave above is not the original definition from that paper, but rather a simplification due to P.~Johnstone (see \cite{johnstone}, Lemma 3.2). The main motivation for computational density is the following result, that we state but do not prove.

\begin{thm}\label{thm:cd=rightadjoint}
Let $f$ be a \emph{total} applicative morphism. Then:
\begin{itemize}
\item[\textup{(}i\textup{)}]	$f$ is projective and c.d.\@ iff $f$ has a right adjoint in $\sf{pPCA}$;
\item[\textup{(}ii\textup{)}]	$f$ is c.d.\@ iff $\asm(f)$ has a right adjoint, iff $\rt(f)$ has a right adjoint.
\end{itemize}
In particular, every total c.d.\@ applicative morphism $A\to B$ gives rise to geometric morphism $\rt(B)\to \rt(A)$.
\end{thm}
\begin{rem}
Even though we do not prove \cref{thm:cd=rightadjoint}, we will make some remarks on the proof method. The right adjoint of $\asm(f)$, if it exists, is certainly a $\nabla$-functor, but it is not necessarily regular. But we do know that it is left exact, and the paper \cite{faberjaap} (Theorem 2.2) shows, for the absolute case, that such functors arise from applicative morphisms $TA\to B$. The assumption that $f$ is c.d.\@ can then be used to construct this morphism, which yields the desired right adjoint of $\asm(f)$. This construction also works for the relative case, except that one may get a \emph{partial} applicative morphism $TA\pf B$. This corresponds to the fact that the right adjoint of $\asm(f)$ may fail to be a $\Gamma$-functor. As for (i), we have that $f$ has a right adjoint in $\sf{pPCA}$ iff the right adjoint of $\asm(f)$ is regular. This is equivalent to: $\asm(f)$ preserve projective objects, which is equivalent to $f$ being projective (hence the name).
\end{rem}

\begin{ex}\label{ex:iota_r_cd}
The morphism $\iota_r$ from \cref{ex:iota_r} is c.d. Indeed, define the element $n:=\lambda^\ast x.xr\in F_r$, where $F_r = \langle A^\#\cup\{r\}\rangle$ is as in \cref{ex:filter=PCA}. If $s\in F_r$, then by \cref{lem:generate}, there exists a term $t(\vec{x},y)$ without constants such that $t(\vec{a},r)\leq s$ for certain $\vec{a}\in A^\#$. Now consider the element $q=\lambda^\ast y. t(\vec{a},y)\in A^\#$. We have $nq \preceq qr \preceq t(\vec{a},r) \leq s$, which implies $n\cdot \iota_r(q)\subseteq\downset\{s\}$, so $n$ satisfies \cref{eq:cd}.

We could also have shown that $\iota_r$ is c.d.\@ by exhibiting a right adjoint in $\sf{PCA}$. Indeed, consider $h\colon A\to A$ defined by $h(a) = \{b\in A\mid br\leq a\}$. First of all, this is clearly a downset of $A$, and we have $\sf{k}a\in h(a)$ for all $a\in A$, and $\sf{k}a\in h(a)\cap A^\#$ for all $a\in A^\#$. Moreover, $h$ preserves the order on the nose, and finally, $h$ is tracked by $\sf{s}\in A^\#$. We also have $r\in\dom^\# h$, since $\sf{i}\in h(r)\cap A^\#$, so $h$ is also a morphism $A[r]\to A$, by the universal property of $\iota_r$. It is easy to see that $\sf{k}\in A^\#$ realizes both $\id_A\leq \iota_r h$ and $\id_{A[r]}\leq h\iota_r$, whereas $\lambda^\ast x.xr\in F_r$ realizes $\iota_rh\leq\id_{A[r]}$. So we have an adjunction $\iota_r\dashv h$ in $\sf{PCA}$ with $\iota_r h\simeq \id_{A[r]}$. Applying the construction $\rt$ now shows that $\rt(A[r])$ is a \emph{subtopos} of $\rt(A)$. In fact, one can show (e.g., using techniques from \cite{Z20}) that $\rt(A[r])$ is an \emph{open} subtopos, given by the subterminal assembly $1_r$ defined by $|1_r| = \{\ast\}$ and $E_{1_r}(\ast) = \downset\{r\}$.
\end{ex}

The remainder of the paper can be seen as an attempt to generalize this example to `higher orders'.


\section{Adjoining a partial function to a PCA}\label{sec:order1}

In this section, we generalize a construction from \cite{Af}, which, given a (discrete, absolute) PCA $A$ and a partial function $f\colon A\pf A$, `freely adjoins' $f$ to $A$. In fact, this can be seen as a higher-order version of the construction of $A[r]$, where we view $A[r]$ as the zeroth-order case (adjoining an element), and the construction in this section as the first-order case. Accordingly, we will denote the resulting PCA by $A[f]$.

Here, of course, we treat the relative, ordered case. Since the order is `hard-wired' into the PCA, we will not consider all partial functions on $A$, but only those that cooperate well with the order, i.e., those from $\mathcal{B}A$. The filter is a bit more flexible. For ordinary computability on $\mathbb{N}$, we know that declaring a certain function to be computable (e.g., by making it into an oracle) causes other functions to be computable as well. Because we make a distinction between computable and non-computable \emph{elements} as well, there is also the possibility that more elements become computable, i.e., that the filter becomes larger. More precisely, we will need to close $A^\#$ under application of $f$. Here a set $X\subseteq A$ is closed under application of $f$ if $f(a)\in A$ whenever $a\in X$ and $f(a)\denotes$.

Before we proceed to define $A[f]$, let us first give a precise meaning to the computability of a function $f\in \mathcal{B}A$.

\begin{defn}\label{defn:rep}
Let $A$ be a PCA and let $f\in \mathcal{B}A$.
\begin{itemize}
\item[(i)]	An $r\in A$ is said to \emph{represent} $f$ if $ra\preceq f(a)$ for all $a\in A$. The function $f$ is called \emph{representable} if it represented by an $r\in A$, and \emph{effectively representable} if it is represented by an $r\in A^\#$.
\item[(ii)]	If $g\colon A\pf B$ is a partial applicative morphism, then we say that an $s\in B$ \emph{represents $f$ w.r.t.\@ $g$} if $s\cdot g(a)$ is defined and a subset of $g(f(a))$ for all $a\in \dom f$. The function $f$ is called \emph{representable w.r.t.\@ $g$} if it is represented w.r.t.\@ $g$ by an $s\in B$, and \emph{effectively representable w.r.t.\@ $g$} if it is represented w.r.t.\@ $g$ by an $s\in B^\#$. 
\end{itemize}
\end{defn}
Observe that (i) is actually a special case of (ii), if we let $g$ be $\id_A$.
\begin{rem}\label{rem:rep}
If $r\in A$, then the partial function $\lambda a.ra$ is always in $\mathcal{B}A$. (Here we mean $\lambda a.ra$ to denote the \emph{partial function} that sends $a$ to $ra$, if defined; not to be confused with $\lambda^\ast a.ra\in A$, which represents this function!) This means that the set of all respresentable resp.\@ effectively representable $f$ can also be described as
\[
\upset\{\lambda a.ra\mid r\in A\}\subseteq \mathcal{B}A \quad\mbox{resp.}\quad\upset\{\lambda a.ra\mid r\in A^\#\}\subseteq \mathcal{B}A.
\]
This will become important in the next section. 
\end{rem}

We observe that effective representability is transferable along partial applicative morphisms in the following sense: if $A\stackrel{g}{\pf} B\stackrel{h}{\pf} C$ and $f\in\mathcal{B}A$ is effectively representable w.r.t.\@ $g$, then $f$ is also effectively representable w.r.t.\@ $hg$.
In order to prove this, assume for simplicity that $h$ preserves the order on the nose, and let $t\in C^\#$ track $h$. If $s\in B$ represents $f$ w.r.t.\@ $g$, then $\lambda^\ast x. ts_0x\in C^\#$ represents $f$ w.r.t.\@ $hg$, where $s_0\in h(s)\cap C^\#$. A similar argument shows that representability is transferable along \emph{total} applicative morphisms.

We will now begin the construction of the desired `free' PCA $A[f]$ is which $f$ is effectively representable. In order for this construction to work, we need to assume that $A$ is not semitrivial. In fact, from this point onwards, we will assume that all PCAs we consider are not semitrivial. The underlying set of $A[f]$ will simply be $A$ itself, and the order on $A[f]$ will simply be the order on $A$. However, we equip $A[f]$ with a new application operation. Informally, a computation in $A[f]$ will be a computation in $A$ with an oracle for $f$. That is, the computation can feed a finite number of inputs to $f$ before coming up with the final result. In order to distinguish this new application from the original one, we will write it as $a\odot b$. Of course, this new application will depend on $f$, so really, we should write $\odot_f$. If there is ambiguity as to which function plays the role of $f$, we will do so. However, doing so consistently would make expressions concerning $A[f]$ rather illegible, so in this section, we exclusively write $\odot$.

\begin{defn}\label{defn:Af_PAP}
Let $A$ be a PCA and let $f\in\mathcal{B}A$. We define the PAP $A[f] = (A,\odot,\leq)$ as follows. For $a,b,c\in A$, we say that $a\odot b = c$ if and only if there exists a (possibly empty) sequence $u_0, \ldots, u_{n-1}\in A$ such that
\begin{itemize}
\item	for all $i<n$, we have $\sf{p}_0(a\cdot[b,u_0, \ldots, u_{i-1}])\leq \bot$ and $f(\sf{p}_1(a\cdot[b,u_0, \ldots, u_{i-1}])) = u_i$;
\item	$\sf{p}_0(a\cdot[b,u_0, \ldots, u_{n-1}])\leq \top$ and $\sf{p}_1(a\cdot[b,u_0, \ldots, u_{n-1}]) = c$.
\end{itemize}
The sequence $u_0, \ldots, u_{n-1}$ is called a \emph{$b$-interrogation of $f$ by $a$}.
\end{defn}
Intuitively, the coefficients in the interrogation are the values the oracle returns in the course of the computation of $a\odot b$. At each stage of the computation, the algorithm $a$ is allowed to consult the input $b$ and the values $u_0, \ldots, u_{i-1}$ obtained from the oracle so far. Formally, this means that we let $a$ act on the coded sequence $[b,u_0, \ldots, u_{i-1}]$. We view the result as carrying two pieces of information. The first piece is a (sub)boolean, which tells us whether the computation has gathered enough oracle values to output a result. If not, then the second piece of information is fed to the oracle; if the oracle need not be consulted anymore, then this second piece is the output.

Since $A$ is not semitrivial, there is at most one $b$-interrogation of $f$ by $a$, which also means that $a\odot b = c$ for at most one $c\in A$. Observe that $a\odot b$ may fail to be defined in several ways. First of all, one of the applications in $A$ could be undefined. In addition, $\sf{p}_0(a\cdot[b,u_0, \ldots, u_{i-1}])$ could fail to be a subboolean or $\sf{p}_1(a\cdot[b,u_0, \ldots, u_{i-1}])$ could lie outside the domain of $f$ (i.e., the oracle fails to return a value). Finally, it could happen that the computation keeps feeding inputs to the oracle indefinitely, never coming up with a final output. For example, if $a = \sf{k}(\sf{p}\bot\bot)$, then $a\odot b$ will always be undefined, even if $A$ itself is a total PCA and $f$ is a total function.

\begin{rem}\label{rem:original_Af}
In the original definition of $a\odot b$ from \cite{Af}, which is for discrete PCAs, the sequence $u_0, \ldots, u_{n-1}$ should satisfy:
\begin{itemize}
\item	for all $i<n$, there exists a $v_i\in A$ such that $a\cdot[b, u_0, \ldots, u_{i-1}] = \sf{p}\bot v_i$ and $f(v_i) = u_i$;
\item	there exists a $c\in A$ such that $a\cdot [b, u_0, \ldots, u_{n-1}] = \sf{p}\top c$,
\end{itemize}
in which case $a\odot b = c$. Since we are working with ordered PCAs, however, we cannot hope to get \emph{equalities} between elements from $A$, since all the available combinators only yield inequalities. We do have the following, which we will use most often when computing $a\odot b$. If there are $\bar{u}_0, \ldots, \bar{u}_{n-1},c\in A$ such that:
\begin{itemize}
\item	for all $i<n$, there exists a $v_i\in A$ such that $a\cdot[b, \bar{u}_0, \ldots, \bar{u}_{i-1}] \leq \sf{p}\bot v_i$ and $f(v_i) \leq \bar{u}_i$;
\item	$a\cdot [b, \bar{u}_0, \ldots, \bar{u}_{n-1}] \leq \sf{p}\top c$,
\end{itemize}
then $a\odot b\leq c$. (We write $\bar{u}_i$ rather than $u_i$ because this sequence need not be the actual $b$-interrogation of $f$ by $a$.)
\end{rem}

Of course, we should show that $A[f]$ is actually a PAP, which is also the point where we need that $f\in\mathcal{B}A$. Suppose we have $a'\leq a$ and $b'\leq b$ such that $a\odot b\denotes$, and let $u_0, \ldots, u_{n-1}$ be the $b$-interrogation of $f$ by $a$. Then by induction, one easily shows that there exist $u'_i\leq u_i$ such that $u'_0, \ldots, u'_{n-1}$ is a $b'$-interrogation of $f$ by $a'$; and from this, we get that $a'\odot b'$ is defined and $a'\odot b'\leq a\odot b$, as desired.

In order to complete the definition of $A[f]$ as a PAS, it remains to define $A[f]^\#$. If $A[f]$ is to be the free PCA in which $f$ is effectively representable, there should be a morphism $A\to A[f]$ (cf.\@ the zeroth-order case). The identity on $A$, i.e., $a\mapsto\downset\{a\}$, is an obvious candidate, and in order for this to be a morphism, we must have $A^\#\subseteq A[f]^\#$. We cannot, in general, let $A[f]^\#$ be \emph{equal} to $A^\#$, however, since $A^\#$ could fail to be closed under (defined) $\odot$. The following definition remedies this.
\begin{defn}
Let $A$ be a PCA and let $f\in\mathcal{B}A$. The PAP $A[f]$ is made into a PAS by setting $A[f]^\#:=\langle A^\#\rangle$, where the generated filter is taken in the PAP $A[f]$, rather than $A$, of course.
\end{defn}
This, by definition, makes $A[f]$ into a PAS.

\begin{prop}\label{prop:Af_PCA}
For each PCA $A$ and $f\in\mathcal{B}A$, the quadruple $A[f] = (A,A[f]^\#,\odot,\leq)$ is a PCA.
\end{prop}
\begin{proof}
We need to exhibit suitable combinators $\sf{k}_f$ and $\sf{s}_f$ for $A[f]$. Recall the combinator $\sf{fst}\in A^\#$ from \cref{defn:rij}. For $\sf{k}_f$, we can take $\lambda^\ast x.\sf{p}\top(\lambda^\ast y.\sf{p}\top (\sf{fst}x))\in A^\#\subseteq A[f]^\#$. Indeed, if $a,b\in A$, then $\sf{k}_f\odot a\leq (\lambda^\ast y.\sf{p}\top (\sf{fst}x))[[a]/x]$, and
\[
(\sf{k}_f\odot a)\cdot[b]\preceq (\sf{p}\top(\sf{fst}x))[[a]/x,[b]/y] = \sf{p}\top(\sf{fst} [a]) \leq \sf{p}\top a,
\]
which means that $\sf{k}_f\odot a\odot b\leq a$, as desired. Observe that $\sf{k}_f$ does not, in fact, depend on $f$, and that the computation of $\sf{k}_f\odot a\odot b$ does not consult the oracle at all.

The definition of $\sf{s}_f$ (which will also not depend on $f$) is a little more involved, and writing down an actual definition of $\sf{s}_f$ would be quite cumbersome. Therefore, we simply explain how to construct it. Using recursion and the elementary operations on sequences and booleans, we can construct an element $S\in A^\#$ such that for all $a$, $b$ and $u = [u_0, \ldots, u_{n-1}]$ from $A$, we have:
\begin{itemize}
\item	if $\forall i\leq n\s (\sf{p}_0(x\cdot [u_0, \ldots, u_{i-1}])\leq \bot)$, then $Sxyu\preceq xu$;
\item	if $i$ is minimal such that $\sf{p}_0(x\cdot [u_0, \ldots, u_{i-1}])\leq \top$, and 
\[\forall j\leq n\s(i\leq j\to \sf{p}_0(y\cdot [u_0, u_i, \ldots, u_{j-1}])\leq \bot),\] then $Sxyu\preceq y\cdot [u_0, u_i, \ldots, u_{n-1}]$;
\item	if $i$ is minimal such that $\sf{p}_0(x\cdot [u_0, \ldots, u_{i-1}])\leq \top$, and $j\geq i$ is minimal such that $\sf{p}_0(y\cdot [u_0, u_i, \ldots, u_{j-1}])\leq \top$, then 
\[
Sxyu\preceq \sf{p}_1(x\cdot[u_0, \ldots, u_{i-1}])\cdot[\sf{p}_1(y\cdot[u_0, u_i, \ldots, u_{j-1}]),u_j,\ldots, u_{n-1}].
\]
\end{itemize}
We leave it to the reader to check that $Sab\odot c\preceq a\odot c\odot (b\odot c)$ for $a,b,c\in A$. (This is a good exercise in understanding what the definition of $S$ above actually does!) Finally, we can set
\[\sf{s}_f = \lambda^\ast x. \sf{p}\top(\lambda^\ast y. S(\sf{fst}x)(\sf{fst}y))\in A^\#\subseteq A[f]^\#.
\]
In the same way as we did for $\sf{k}_f$, one can verify that $\sf{s}_f\odot a\odot b\leq Sab$, so we can conclude that $A[f]$ is a PCA.
\end{proof}

Before we continue, we introduce two algorithms $t_f,r_f\in A^\#\subseteq A[f]^\#$ that are relevant throughout this section. Set
\[
t_f := \lambda^\ast x.\sf{p}\top(\lambda^\ast y.\sf{p}\top(\sf{fst}x(\sf{fst}y))),
\]
where $\sf{fst}$ is as above. A calculation similar to the ones in the proof above above shows that $t_f\odot a\odot b\preceq a\cdot b$, for all $a,b\in A$. Moreover, set
\[
r_f := \lambda^\ast x.(\sf{if}\ \sf{zero}(\sf{pred}(\sf{lh}x))\ \sf{then}\ \sf{p}\bot (\sf{fst}x)\ \sf{else}\ \sf{p}\top (\sf{read}x1)).=,
\]
where $\sf{lh}, \sf{fst}, \sf{read}\in A^\#$ are as in \cref{defn:rij}. If $a\in A$ is such that $f(a)\denotes$, then
\begin{itemize}
\item	$r_f\cdot [a]\leq \sf{p}\bot a$ and $f(a)\denotes$;
\item	$r_f\cdot [a,f(a)]\leq \sf{p}\top f(a)$,
\end{itemize}
which means that $r_f\odot a\leq f(a)$. So we have $r_f\odot a\preceq f(a)$ for all $a\in A$.

In analogy with \cref{ex:iota_r}, we define $\iota_f\colon A\to A[f]$ by $\iota_f(a) = \downset\{a\}$.
\begin{prop}\label{prop:iotaf}
The map $\iota_f$ is a total applicative morphism, and $f$ is effectively representable w.r.t.\@ $\iota_f$. Moreover, $\iota_f$ has a right adjoint $h\colon A[f]\to A$ satisfying $\iota_fh\simeq \id_{A[f]}$. In particular, $\iota_f$ is c.d.\@ and decidable.
\end{prop}
\begin{proof}
Since $A^\#\subseteq A[f]^\#$, it is clear that $\iota_f$ satisfies the first requirement; and it is also obvious that $\iota_f$ preserves the order on the nose. Moreover, $t_f$ is a tracker, so $\iota_f$ is indeed a total applicative morphism, and $r_f$ represents $f$ w.r.t.\@ $\iota_f$.

We define the required right adjoint $h\colon A[f]\to A$ by:
\[
h(a) := \{b\in A\mid b\odot\sf{i} \leq a\},
\]
where $\sf{i}\in A^\#$ is the identity combinator for $A$. It is clear that $h(a)$ is a downset, and that $h$ preserves the order on the nose. Now consider the $S\in A^\#$ constructed in the proof of \cref{prop:Af_PCA}. If $b\in h(a)$ and $b'\in h(a')$, then $Sbb'\odot \sf{i}\preceq (b\odot \sf{i})\odot(b'\odot \sf{i}) \preceq a\odot a'$, so if $a\odot a'\denotes$, then $Sbb'\in h(a\odot a')$. In other words, $S$ is a tracker for $h$. Next, we observe that $\sf{p}\top a\in h(a)$ for all $a\in A$, so in particular, we have $A^\#\subseteq\dom^\# h$. Since $h$ preserves the order and has a tracker, we already know that $\dom^\# h$ is a filter of the partial applicative \emph{preorder} $A[f]$. Combining this with $A^\#\subseteq \dom^\# h$ yields $A[f]^\#\subseteq \dom^\# h$, so $h$ is a total applicative morphism.

By the observation above, we know that $\sf{p}\top\in A^\#$ realizes $\id_A\leq h\iota_f$. Finally, $\id_{A[f]}\leq \iota_fh$ and $\iota_fh\leq\id_{A[f]}$ are realized by $\sf{k}_f$ and $\lambda^\ast x. x\odot \sf{i}$, respectively.
\end{proof}

Since $\iota_f$ is decidable, we see that the boolean $\bot,\top\in A^\#$ of $A$ can also serve as booleans in $A[f]$. In particular, $A[f]$ is also not semitrivial, given that $A$ is not semitrivial.

In analogy with \cref{ex:iota_r_cd}, we see that $\rt(A[f])$ is a subtopos of $\rt(A)$. Before we proceed to establish the universal property of $A[f]$, we give an alternative description of $A[f]^\#$.
\begin{lem}\label{lem:alt_Af}
Let $A$ be a PCA and let $f\in \mathcal{B}A$. Then $A[f]^\#$ is the least filter on $A$ which is closed under application of $f$.
\end{lem}
\begin{proof}
We need to show the following: if $A^\#\subseteq X\subseteq A$ and $X$ is upwards closed, then $X$ is closed under defined $\odot$ if and only if $X$ is closed under application in $A$ and application of $f$.

First of all, suppose that $X$ is closed under defined $\odot$. If we have $a,b\in X$ such that $ab\!\denotes$, then $t_f\odot a\odot b\leq ab$, so since $t_f,a,b\in X$, we have that $t_f\odot a\odot b\in X$ as well, hence $ab\in X$. Similarly, if $a\in X$ is such that $f(a)\denotes$, then $r_f\odot a\leq f(a)$ yields that $f(a)\in X$ as well.

Conversely, suppose that $X$ is closed under application in $A$ and application of $f$. Suppose that we have $a,b\in X$ such that $a\odot b\denotes$, and let $u_0, \ldots, u_{n-1}$ be the $b$-interrogation of $f$ by $a$. Then using induction, we may show that $u_i\in X$ for all $i<n$. Indeed, let $i<n$ and suppose that $u_0, \ldots, u_{i-1}\in X$. Then $\sf{p}_1(a\cdot [b, u_0, \ldots, u_{i-1}])$, which is an expression built using $a,b,u_0, \ldots, u_{i-1}\in X$, combinators from $A^\#\subseteq X$ and application, must be in $X$ as well. Since $X$ is closed under $f$, it follows that $u_i\in X$ as well, completing the induction. Finally, we see that $a\odot b = \sf{p}_1(a\cdot [b, u_0, \ldots, u_{n-1}])$ is in $X$ as well, completing the proof.
\end{proof}

\begin{thm}\label{thm:universal_Af}
Let $g\colon A\pf B$ be a decidable partial applicative morphism, and let $f\in\mathcal{B}A$. Then $g$ factors, up to isomorphism, through $\iota_f$ if and only if $f$ is effectively representable w.r.t.\@ $g$.
\end{thm}
\begin{proof}
The `only if' statement is clear, since $f$ is effectively representable w.r.t.\@ $\iota_f$, and effective representability transfers along partial applicative morphisms. Conversely, suppose that $f$ is effectively representable w.r.t.\@ $g$. We need to show that $g$ is also a partial applicative morphism $A[f]\pf B$; then we will have that the triangle
\[
\begin{tikzcd}
A \arrow[rr, "g"] \arrow[rd, "\iota_f"'] &                         & B \\
                                         & {A[f]} \arrow[ru, "g"'] &  
\end{tikzcd}
\]
commutes up to isomorphism. Assume for simplicity that $g$ preserves the order on the nose, let $t\in B^\#$ be a tracker of $f$, let $d\in B^\#$ be a decider for $g$, and let $s\in B^\#$ represent $f$ w.r.t.\@ $g$.

First, we need to show that $A[f]^\#\subseteq \dom^\# g$. We already know that $\dom^\# g$ is a filter of the PCA $A$, so according to \cref{lem:alt_Af}, it suffices to show that $\dom^\# g$ is closed under application of $f$. So suppose that $a\in\dom^\# g$ and that $f(a)\denotes$. Then there exists a $b\in g(a)\cap B^\#$, and we get $sb\in g(f(a))\cap B^\#$, so $f(a)\in\dom^\# g$ as well, as desired. Of course, $g$ still preserves the order when considered as a morphism $A[f]\pf B$, so it remains to construct a tracker.

First, recall the combinators $\sf{unit},\sf{ext}\in A^\#$ from \cref{defn:rij}. For $i=0,1$, define $\sf{p}'_i = \lambda^\ast x. tp_ix\in B^\#$, where $p_i$ is any element from $g(\sf{p}_i)\cap B^\#$. This element has the property that $\sf{p}'_i\cdot g(a)\preceq g(\sf{p}_ia)$ for $a\in A$. Similarly, using an element from $g(\sf{unit})\cap B^\#$, we define $\sf{unit}'\in B^\#$ such that $\sf{unit}'\cdot g(a)\subseteq g(\sf{unit}\cdot a)\subseteq g([a])$ for all $a\in A$. Moreover, we may define $\sf{ext}'\in B^\#$ such that $\sf{ext}'\cdot g([a_0, \ldots, a_{n-1}])\cdot g(a')\subseteq g([a_0, \ldots, a_{n-1}, a'])$. Using the fixed point operator in $B^\#$, we can construct an element $T\in B^\#$ satisfying:
\begin{align}
Tbv \preceq\ &\sf{if}\ d(\sf{p}'_0(tbv)))\ \sf{then}\ \sf{p}'_1(tbv)\ \sf{else}\ Tb\big(\sf{ext'}v(s(\sf{p}'_1(tbv)))\big).\label{eq:defT}
\end{align}
Suppose that $a,a'\in A$ are such that $a\odot a'\denotes$, and let $u_0, \ldots, u_{n-1}$ be the $a'$-interrogation of $f$ by $a$. First of all, we claim that
\begin{align}\label{eq:induction_step}
T\cdot g(a)\cdot g([a',u_0, \ldots, u_{i-1}]) \preceq T\cdot g(a)\cdot g([a',u_0, \ldots, u_i])
\end{align}
for all $i<n$. Suppose that the right hand side of \cref{eq:induction_step} is defined, and consider $b\in g(a)$ and $v\in g([a',u_0, \ldots, u_{i-1}])$. Then we have $tbv\in g(a\cdot [a',u_0, \ldots, u_{i-1}])$, so 
\[
\sf{p}'_0(tbv)\in g(\sf{p}_0(a\cdot [a',u_0, \ldots, u_{i-1}])) \subseteq g(\bot).\]
This gives $d(\sf{p}'_0(tbv)) \leq \bot$, so we need to evaluate the `else' clause in \cref{eq:defT}. We have $\sf{p}'_1(tbv) \in g(\sf{p}_1(a\cdot [a',u_0, \ldots, u_{i-1}]))$, which gives $s(\sf{p}'_1(tbv))\in g(f(\sf{p}_1(a\cdot [a',u_0, \ldots, u_{i-1}]))) = g(u_i)$. Now $\sf{ext}'v(s(\sf{p}'_1(tbv)))\in g([a',u_0, \ldots, u_i])$, and \cref{eq:defT} tells us that $Tbv\leq Tb\big(\sf{ext'}v(s(\sf{p}'_1(tbv)))\big) \in T\cdot g(a)\cdot g([a',u_0, \ldots, u_i])$, as desired.

Moreover, we have:
\begin{align}\label{eq:endgame}
T\cdot g(a)\cdot g([a',u_0, \ldots, u_{n-1}]) \subseteq g(a\odot a').
\end{align}
Indeed, consider $b\in g(a)$ and $v\in g([a',u_0, \ldots, u_{n-1}])$. Then as above, we find $d(\sf{p}'_0(tbv)))\leq\top$, so \cref{eq:defT} tells us that $Tbv\leq \sf{p}'_1(tbv)\in g(\sf{p}_1(a\cdot [a',u_0, \ldots, u_{n-1}])) = g(a\odot a')$. Combining \cref{eq:induction_step} and \cref{eq:endgame} yields:
\[
T\cdot g(a) \cdot g([a']) \subseteq g(a\odot a')
\]
whenever $a\odot a'\denotes$. We conclude that $\lambda^\ast xy. Tx(\sf{unit}'y)$ tracks $g\colon A[f]\pf B$, which finishes the proof.
\end{proof}

Since identities are c.d., and c.d.\@ partial applicative morphisms are closed under composition, there is a wide subcategory $\sf{pPCA}_\text{cd}$ of $\sf{pPCA}$ consisting of only the c.d.\@ partial applicative morphisms. Similarly, we have the subcategory $\sf{pPCA}_\text{dec}$ consisting of only the decidable morphisms. We define $\sf{PCA}_\text{cd}$ and $\sf{PCA}_\text{dec}$ analogously. The following now easily follows from \cref{thm:universal_Af}.

\begin{cor}\label{cor:universal_Af}
Let $A,B$ be PCAs, let $f\in\mathcal{B}A$ and let $\mathcal{C}$ be any of the preorder-enriched categories $\sf{pPCA}_\textup{cd}, \sf{pPCA}_\textup{dec}, \sf{PCA}_\textup{cd}, \sf{PCA}_\textup{dec}$. Then composition with $\iota_f$:
\[
\mathcal{C}(A[f],B) \to \{g\in\mathcal{C}(A,B)\mid f\mbox{ is effectively representable w.r.t.\@ }g\}
\]
is an equivalence of preorders.
\end{cor}

\begin{proof}
\cref{thm:universal_Af} readily implies that for each of the $\mathcal{C}$, the map above is essentially surjective. Moreover, composition with $\iota_f$ reflects the order since $\iota_f$ has a right pseudoinverse as shown in \cref{prop:iotaf}.
\end{proof}

\begin{ex}
The construction of $A[f]$ is a generalization of oracle computations for classical Turing computability. Indeed, if $f\colon\mathbb{N}\pf \mathbb{N}$ is a partial function, then $g\colon \mathbb{N}\pf\mathbb{N}$ is effectively representable w.r.t.\@ $\iota_f\colon\mathcal{K}_1\to\mathcal{K}_1[f]$ iff $g$ is Turing computable relative to an oracle for $f$. See also Corollary 2.3 in \cite{Af}.
\end{ex}

\begin{ex}\label{ex:r_is_rhat}
As we mentioned at the beginning, the construction from this section can be seen as a higher-order version of the construction from \cref{ex:iota_r}. On the other hand, the construction of $A[r]$ can be seen as a special case of the construction $A[f]$. Indeed, consider $r\in A$ and denote the constant function with value $r$, which is an element of $\mathcal{B}A$, by $\hat{r}$. It is easy to see that, for any $g\colon A\pf B$, we have that $\hat{r}$ is effectively representable w.r.t.\@ $g$ iff $r\in\dom^\# g$. It follows that $A[r]$ and $A[\hat{r}]$ are equivalent PCAs.
\end{ex}

\begin{ex}\label{ex:Af_to_Ar}
Of course, if $f\in\mathcal{B}A$ is already effectively representable in $A$ itself, then $\iota_f\colon A\to A[f]$ will be an isomorphism of PCAs. Now suppose that $f$ is represented by an element $r\in A$ (but not necessarily $r\in A^\#$). Then $f$ is effectively representable w.r.t.\@ $\iota_r\colon A\to A[r]$, so we get a factorization:
\[\begin{tikzcd}
A \arrow[d, "\iota_f"'] \arrow[rd, "\iota_r"] &        \\
{A[f]} \arrow[r, dashed]                      & {A[r]}
\end{tikzcd}\]
It is worth observing that the mediating arrow $A[f]\to A[r]$ is \emph{not}, in general, an isomorphism. Indeed, consider, e.g., a PCA with $A^\#\neq A$, and take a $b\in A\backslash A^\#$. Then the partial function $f\in\mathcal{B}A$ defined by 
\[
f(a) = \begin{cases}
\sf{p}_1a &\mbox{if }\sf{p}_0a\leq\top;\\
\mbox{undefined} &\mbox{else}
\end{cases}
\]
is effectively representable, e.g., by $\sf{p}_1\in A^\#$. This means that $\iota_f$ is an isomorphism. But it is \emph{also} representable by $r:=\lambda^\ast x.\sf{if}\ \sf{p}_0x\ \sf{then}\ \sf{p}_1x\ \sf{else}\ b$. Moreover, we have $r\not\in A^\#$, because $r(\sf{p}\bot\bot)\leq b$, so $r\in A^\#$ would imply $b\in A^\#$. This means that $\iota_r$ is \emph{not} an isomorphism, so $A[f]\to A[r]$ cannot be an isomorphism either. So we see that the point here is, really, that a function $f\in\mathcal{B}A$ can have many representers.
\end{ex}


\section{The PCA of partial functions}\label{sec:BA}

In this section, we show how to turn the set $\mathcal{B}A$ from \cref{sec:ROPCA} into a PCA. As the order, we use the order defined in \cref{defn:BA_as_poset}(ii). In particular, the empty function is a largest element of $\mathcal{B}A$. It is worth noting that, in contrast with $A[f]$, the order on $\mathcal{B}A$ is \emph{not} discrete even if the order on $A$ is. Indeed, if $A$ is discrete, then $\mathcal{B}A$ consists of all partial functions $A\pf A$, and the order is the \emph{reverse} subfunction relation. In this case, the total functions are minimal elements of $\mathcal{B}A$. In the general case, the total functions in $\mathcal{B}A$ form a downwards closed set.

The application on $\mathcal{B}A$ will, in a sense, generalize the $A[f]$ for $f\in\mathcal{B}A$ all at once. As for the construction of $A[f]$, we need the assumption that $A$ is not semitrivial. An important thing to note about the application on $\mathcal{B}A$ is that it will be \emph{total}.

Let us define the application now. For $\alpha,\beta\in\mathcal{A}$ and $a,b\in A$, we say that $\alpha\beta(a) = b$ if and only if there are $u_0, \ldots, u_{n-1}$ such that:
\begin{itemize}
\item	for all $i<n$, we have $\sf{p}_0\cdot\alpha([a,u_0, \ldots, u_{i-1}]) \leq \bot$ and $\beta(\sf{p}_1\cdot\alpha([a,u_0, \ldots, u_{i-1}])) = u_i$;
\item	$\sf{p}_0\cdot\alpha([a,u_0, \ldots, u_{n-1}]) \leq \top$ and $\sf{p}_1\cdot\alpha([b,u_0, \ldots, u_{i-1}]) = b$.
\end{itemize}
By the assumption that $A$ is nontrivial, the sequence $u_0, \ldots, u_{n-1}$ is unique if it exists, and if it exists it is called the $a$-interrogation of $\beta$ by $\alpha$. The following table compares the definition of $\alpha\beta(a)$ with the definition of $a\odot_f b$ from \cref{defn:Af_PAP}.
\begin{center}
\begin{tabular}{|l|c|c|}
\hline & $a\odot_f b$ & $\alpha\beta(a)$\\
\hline Interrogator & $a$ & $\alpha$ \\
\hline Input & $b$ & $a$ \\
\hline Oracle & $f$ & $\beta$ \\
\hline
\end{tabular}
\end{center}
The above defines $\alpha\beta$ as a partial function $A\pf A$; we leave it to the reader to check that $\alpha\beta$ is actually in $\mathcal{B}A$, and that this application makes $\mathcal{B}A$ into a PAP. The argument is very similar to the proof that $A[f]$ is a PAP. It should be noted that, even though $\alpha\beta$ is always defined, $\alpha\beta(a)$ could be undefined for the same reasons $a\odot_f b$ could be undefined. Moreover, a remark similar to \cref{rem:original_Af} applies: the PCA $\mathcal{B}A$ was introduced in \cite{PCAfunctions} for discrete PCAs, where one finds a definition of application involving equalities between elements of $A$. Again, this version of the definition is not suitable for the ordered case.

For $(\mathcal{B}A)^\#$, we take the set of all effectively representable functions from $\mathcal{B}A$. By \cref{rem:rep}, this is equivalent to saying:
\[
(\mathcal{B}A)^\# = \upset\{\lambda a.ra\mid r\in A^\#\}.
\]
This set is clearly upwards closed, so we need to check that it also closed under the application defined above. This requires a bit more work. Suppose that we have $\rho,\sigma\in (\mathcal{B}A)^\#$, and take $r,s\in A^\#$ such that $ra\preceq\rho(a)$ and $sa\preceq \sigma(a)$. Using the fixpoint operator, we may find a $T\in A^\#$ such that
\begin{align*}
Txyu \preceq\ \sf{if}\ \sf{p}_0(xu) \ \sf{then}\ \sf{p}_1(xu) \ \sf{else}\ Txy(\sf{ext}\cdot u\cdot (y(\sf{p}_1u))),
\end{align*}
where $\sf{ext}\in A^\#$ is as in \cref{defn:rij}. Now it is not hard to check that $Trs\cdot[a]\preceq \rho\sigma(a)$, which means that $\rho\sigma$ is represented by $\lambda^\ast x. Trs(\sf{unit}\cdot x)\in A^\#$. We can conclude that $\rho\sigma\in A^\#$ as well, so $\mathcal{B}A$ is a PAS.

Showing that $\mathcal{B}A$ is a PCA is rather involved, and we will not provide all the details here. In fact, the $\sf{k}$-combinator is still easy: using operations on sequences, one may construct an $r\in A^\#$ such that:
\begin{itemize}
\item	$r\cdot [[[a]]] \leq \sf{p}\bot a$;
\item	$r\cdot [[[a],b]] \leq \sf{p}\top(\sf{p}\top(\sf{p}\top b))$.
\end{itemize}
Then $\boldsymbol{\kappa}\in(\mathcal{B}A)^\#$ defined by $\boldsymbol{\kappa}(a)\simeq ra$ will satisfy $\boldsymbol{\kappa}\alpha\beta(a)\preceq\alpha(a)$, hence $\boldsymbol{\kappa}\alpha\beta\leq\alpha$. The $\sf{s}$-combinator $\boldsymbol{\sigma}\in(\mathcal{B}A)^\#$ can be constructed in the following way. For $\alpha,\beta\in\mathcal{B}A$, the function $\boldsymbol{\sigma}\alpha\beta$ should be an interrogator that, when provided with oracle $\gamma$, works as follows. Suppose that $\alpha\gamma(\beta\gamma)(a)$ is defined with interrogation sequence $u_0,\ldots, u_{n-1}$. On input $[a]$, it first simulates the computation of $\alpha\gamma([a])$. When finished, it checks whether $\sf{p}_0\cdot(\alpha\gamma([a]))$ holds. If so, it can output $\alpha\gamma([a])$. If not, it proceeds to simulate, again using the oracle $\gamma$, the computation $\beta\gamma(\sf{p}_1\cdot \alpha\gamma([a]))$, thus finding $u_0$. Then it proceeds to simulate the computation of $\alpha\gamma([a,u_0])$. It keeps going back and forth between simulating a computation instances of $\alpha\gamma$ and $\beta\gamma$, reconstruction the entire sequence $u_0, \ldots, u_{n-1}$. Then it finally finds that $\sf{p}_0\cdot \alpha\gamma([a,u_0, \ldots, u_{n-1}])$ holds, and it can give the correct output. This means that the task is to construct a $\boldsymbol{\sigma}$ that, when fed oracles for $\alpha$ and $\beta$, produces such an interrogator $\boldsymbol{\sigma}\alpha\beta$. Details on how to do this may be found in \cite{PCAfunctions}; it does not show that $\boldsymbol{\sigma}$ is an effective computation, but the constructions involved are clearly effective computations in $A$. See also the following remark, however.

\begin{rem}\label{rem:why_order}
The paper \cite{PCAfunctions} treats $\mathcal{B}A$ as an absolute, discrete PCA. In this paper, on the other hand, we want to consider $\mathcal{B}A$ as a relative PCA, with $(\mathcal{B}A)^\#$ consisting of the effectively representable functions. By doing so, the applicative morphism $i$ in \cref{prop:constant} below becomes compuationally dense, so we have a geometric morphism $\rt(\mathcal{B}A)\to\rt(A)$, which is even a surjection (\cref{prop:surjection}). Moreover, it allows us to apply the construction from \cref{ex:iota_r} to $\mathcal{B}A$ in a nontrivial way, which is essential for the new topos-theoretic construction of $\rt(A[f])$ described below.

The move to relativity also \emph{forces} us to view $\mathcal{B}A$ as an ordered PCA. Indeed, we do not have that $(\mathcal{B}A,(\mathcal{B}A)^\#,\cdot,=)$ is always a PCA, even if $A$ itself is discrete. What we have, is an effectively representable $\boldsymbol{\sigma}$ such that: if $\alpha\gamma(\beta\gamma)(a)$ is defined, then $\boldsymbol{\sigma}\alpha\beta\gamma(a)$ is also defined  with (in the discrete case) the same value. However, if $\mathcal{B}A$ is equipped with the discrete order, then $\alpha\gamma(\beta\gamma)$ and $\boldsymbol{\sigma}\alpha\beta\gamma$ should be the \emph{same} function. And this is not automatically true. Indeed, one of the reasons $\alpha\gamma(\beta\gamma)$ could turn out undefined, is that a certain intermediate result that should be a boolean, is not in fact a boolean. When simulating computations, $\boldsymbol{\sigma}$ feeds such expressions to the if-then-else operator, and this may very well yield an unintended result. One could remedy this by setting all such unintended results to undefined, or to an output that sends the computation of $\boldsymbol{\sigma}\alpha\beta\gamma(a)$ into an infinite loop, but one cannot hope, in general, that this keeps $\boldsymbol{\sigma}$ effective. It may \emph{happen} to be possible, of course. The most prominent example is Kleene's first model $\mathcal{K}_1$, which allows a nice coding of booleans for which this problem does not arise.

The main focus of \cite{PCAfunctions} is actually not $\mathcal{B}A$, but a related PCA $\mathcal{K}_2A$ whose elements are the \emph{total} functions $A\to A$. The application is defined in the same way, except that $\alpha\beta$ is only defined if $\alpha\beta(a)$ is defined for all $a\in A$. Then the problem described above does not arise, since all the functions involved are total. But there is another problem, namely that the $\boldsymbol{\sigma}$ described above will clearly be \emph{partial}. In order to make $\boldsymbol{\sigma}$ into an element of $\mathcal{K}_2A$, we need to extend it to a total function, and once again, this can not necessarily be done in an effective way. Kleene's first model again forms an exception; indeed, it is well known that Kleene's \emph{second} model, which is $\mathcal{K}_2(\mathcal{K}_1)$, has a relative version.
\end{rem}

For $a\in A$, let $\hat{a}\in\mathcal{B}A$ denote the constant function with value $a$.
\begin{prop}\label{prop:constant}
There is a total c.d.\@ applicative morphism $i\colon A\to \mathcal{B}A$, defined by $i(a) = \downset\{\hat{a}\}$.
\end{prop}
\begin{proof}
If $a\in A^\#$, then $\hat{a}$ is clearly represented by $\sf{k}a\in A^\#$, so $\hat{a}\in (\mathcal{B}A)^\#$, meaning that $i$ satisfies the first requirement. A similar argument shows that $i$ is total. Moreover, $i$ clearly preserves the order on the nose, so in order to show it is an applicative morphism, it remains to construct a tracker. Using the elementary operations on sequences, we may contruct a $t\in A^\#$ satisfying:
\begin{itemize}
\item	$t\cdot [[x]]\leq \sf{p}\top(\sf{p}\bot\sf{i})$;
\item	$t\cdot [[x,y]] \leq \sf{p}\bot\sf{i}$;
\item	$t\cdot [[x,y],z] \preceq \sf{p}\top(\sf{p}\top(zy))$.
\end{itemize}
If $\tau\in(\mathcal{B}A)^\#$ is defined by $\tau(a)\simeq ta$, then it is straightforward to check that $\tau\hat{a}\hat{b}(c)\preceq ab$ for all $a,b,c\in A$, so $\tau$ is a tracker for $i$.

For computational density, we take an $n\in A^\#$ satisfying:
\begin{itemize}
\item	$n\cdot [x]\leq \sf{p}\bot \sf{i}$;
\item	$n\cdot [x,y] \preceq \sf{p}\top(yx)$.
\end{itemize}
If $\nu\in(\mathcal{B}A)^\#$ is defined by $\nu(a)\simeq na$, then it is easy to check that $\nu\hat{r}(a)\preceq ra$ for all $a,r\in A$. So, if $\rho\in(\mathcal{B}A)^\#$ is represented by $r\in A^\#$, then $\nu\hat{r}\leq \rho$, so $\nu\cdot i(r) \subseteq \downset\{\rho\}$, showing that $\nu$ satisfies \cref{eq:cd}.
\end{proof}

Since $i$ is c.d., it is decidable as well, which means that $\widehat{\top},\widehat{\bot}$ can serve as booleans in $\mathcal{B}A$. In particular, $\mathcal{B}A$ is not semitrivial, given that $A$ is not semitrivial.

The fact that $i$ is c.d.\@ is the main reason for viewing $\mathcal{B}A$ as a relative PCA. If we consider $\mathcal{B}A$ as an absolute PCA, then the corresponding result does not hold for cardinality reasons. Indeed, one might say that `there are more functions than respresenters'. The computational density of $i$ means that we get a geometric morphism $\rt(\mathcal{B}A)\to\rt(A)$. In fact, the right adjoint of $i$ must already exist at the level of PCAs, since $i$ is also total and projective. We will construct this right adjoint explicitly.

As in \cite{PCAfunctions}, the PCA $\mathcal{B}A$ has the property that every element of $\mathcal{B}A$ is representable w.r.t.\@ $i\colon A\to \mathcal{B}A$.
\begin{prop}\label{prop:everything_rep}
Let $A$ be a PCA. Then every $\alpha\in \mathcal{B}A$ is representable w.r.t.\@ $i\colon A\to \mathcal{B}A$.
\end{prop}
\begin{proof}
Construct an $r\in A^\#$ satisfying:
\begin{itemize}
\item	$r\cdot [[x]]\leq \sf{p}\top(\sf{p}\bot\sf{i})$;
\item	$r\cdot [[x,y]]\leq \sf{p}\bot y$;
\item	$r\cdot [[x,y],z] \leq \sf{p}\top(\sf{p}\top z)$,
\end{itemize}
and let $\rho\in (\mathcal{B}A)^\#$ be defined by $\rho(a)\simeq ra$. Then it is easily checked that $\rho \alpha\hat{a}\preceq \widehat{\alpha(a)}$ for all $\alpha\in\mathcal{B}A$ and $a\in A$. It follows that $\rho \alpha$ represents $\alpha$ w.r.t.\@ $i$, for all $\alpha\in\mathcal{B}A$.
\end{proof}

The following result also appears as Proposition 5.1 in \cite{PCAfunctions}. Since we have introduced the concept of a \emph{partial} applicative morphism, we can formulate the result here in a nicer way.

\begin{thm}\label{thm:extension_BA}
Let $g\colon A\pf B$ be a decidable partial applicative morphism. Then there exists a largest partial applicative morphism $h\colon \mathcal{B}A\pf B$ such that $hi\simeq g$, and it is explicitly defined by:
\[
h(\alpha) = \{b\in B\mid b\mbox{ represents }\alpha\mbox{ w.r.t.\@ }g\} = \{b\in B\mid \forall a\in A\s (\alpha(a)\denotes\ \to b\cdot g(a)\subseteq g(\alpha(a)))\}.
\]
\end{thm}
\begin{proof}
It is easy to see that $h(\alpha)$ is a downset of $B$. For simplicity, we will assume that $g$ preserves the order on the nose; then it follows that $h$ preserves the order on the nose as well. Moreover, if $\alpha\in (\mathcal{B}A)^\#$, then $\alpha$ is effectively representable in $A$, which implies that $\alpha$ is also effectively representable w.r.t.\@ $g$, which means that $h(\alpha)\cap B^\#$ is non-empty. So in order to show that $h$ is a partial applicative morphism it remains to construct a tracker.

The construction of a tracker will be very similar to the construction of a tracker in the proof of \cref{thm:universal_Af}. Let $t\in B^\#$ be a tracker of $g$ and let $d\in B^\#$ be a decider of $g$. Moreover, take $\sf{ext}', \sf{p}'_0, \sf{p}'_1, \sf{unit}'\in B^\#$ as in the proof of \cref{thm:universal_Af}. Using the fixpoint operator in $B$, we may construct a $U\in B^\#$ such that
\[
Uxyv \preceq\ \sf{if}\ d(\sf{p}'_0(xv)) \ \sf{then}\ \sf{p}'_1(xv) \ \sf{else}\ Uxy(\sf{ext}'v(y(\sf{p}'_1(xv)))).
\]
Using the familiar kind of argument, one can show that $T\cdot h(\alpha) \cdot h(\beta) \cdot g([a]) \subseteq g(\alpha\beta(a))$, whenever $\alpha,\beta\in\mathcal{B}(A)$ and $a\in A$ are such that $\alpha\beta(a)\denotes$. It follows that $\lambda^\ast xyz.Uxy(\sf{unit}'z)\in B^\#$ tracks $h$.

For $a\in A$, we have:
\[
hi(a) = h(\hat{a}) = \{b\in B\mid \forall a'\in A\s (b\cdot g(a')\subseteq g(a))\}.
\]
Now it is easy to see that $\sf{k}\in B^\#$ realizes $g\leq hi$, and if $j\in g(\sf{i})\cap B^\#$, then $\lambda^\ast x. xj\in B^\#$ realizes $hi\leq g$. So we indeed have $hi\simeq g$.

In order to show that $h$ is the \emph{largest} partial appplicative morphism such that $hi\simeq g$, suppose we have another $h'\colon \mathcal{B}A\pf B$ such that $h'i\simeq g$, and assume that $h'$ preserves the order on the nose. Let $t'\in B^\#$ be a tracker of $h'$, and let $r,s\in B^\#$ realize $h'i\leq g$ resp.\@ $g\leq h'i$. Moreover, consider the function $\rho\in (\mathcal{B}A)^\#$ from the proof of \cref{prop:everything_rep}; it has the property that $\rho\alpha\hat{a}\preceq\widehat{\alpha(a)}$ for all $\alpha\in \mathcal{B}A$ and $a\in A$. We pick a $q\in h'(\rho)\cap B^\#$. Now it is easy to see that, if $b\in g(a)$, $c\in h'(\alpha)$ and $\alpha(a)\denotes$, we have $r(t'(t'qc)(sb))\in g(\alpha(a))$. This implies that $\lambda^\ast xy.r(t'(t'qx)(sy))\in B^\#$ realizes $h'\leq h$, as desired.
\end{proof}
Observe that the $h$ constructed above will automatically be decidable as well, since $hi\simeq g$.

\begin{cor}
Let $g\colon A\pf B$ be a decidable partial applicative morphism. Then
\begin{align*}
&\{\alpha\in\mathcal{B}A\mid \alpha\mbox{ is representable w.r.t.\@ }g\}\quad \mbox{and} \\
&\{\alpha\in\mathcal{B}A\mid \alpha\mbox{ is effectively representable w.r.t.\@ }g\}
\end{align*}
are filters of $\mathcal{B}A$.
\end{cor}
\begin{proof}
If we let $h\colon \mathcal{B}A\pf B$ be as in \cref{thm:extension_BA}, then these sets are $\dom h$ and $\dom^\# h$, respectively.
\end{proof}

\begin{ex}\label{ex:id_is_max}
If we apply \cref{thm:extension_BA} to $i$ itself, then we see that
\[
h(\alpha) = \{\beta\in\mathcal{B}A\mid \beta\mbox{ represents }\alpha\mbox{ w.r.t.\@ }i\}
\]
defines the largest $h\colon \mathcal{B}A\to\mathcal{B}A$ such that $hi\simeq i$. We claim that $h$ is isomorpic to the identity. By \cref{thm:extension_BA}, we already know that $\id_{\mathcal{B}A}\leq h$. For the converse inequality, construct an $r\in A^\#$ such that:
\begin{itemize}
\item	$r\cdot [x]\leq \sf{p}\bot[\sf{i}]$;
\item	$r\cdot[x,u_0\ldots,u_i] \leq\ \sf{if}\ \sf{p}_0u_i\ \sf{then}\ \sf{p}\top(\sf{p}_1u_i)\ \sf{else}\ \sf{p}\bot[\sf{i},\underbrace{x, \ldots, x}_{i+1\textup{ times}}]$, for $i\geq 0$.
\end{itemize}
If $\rho\in (\mathcal{B}A)^\#$ is defined by $\rho(a)\simeq ra$, then $\rho\beta(a)\leq \beta\hat{a}(\sf{i})$ for all $a\in A$ and $\beta\in\mathcal{B}A$. In particular, if $\beta\in h(\alpha)$, then $\rho\beta\leq \alpha$, so $\rho$ realizes $h\leq \id_{\mathcal{B}A}$. In a slogan, we could say that \emph{elements} of $\mathcal{B}A$ and their \emph{representers} can be used interchangably in $\mathcal{B}A$.

If $f\in\mathcal{B}A$, then we can consider $\iota_f\colon \mathcal{B}A\to \mathcal{B}A[f]$ as in \cref{ex:iota_r}. A completely similar argument shows that $\iota_f$ is the largest $h\colon \mathcal{B}A\to\mathcal{B}A[f]$ such that $hi\simeq \iota_f i$, so the slogan remains true for this case. If we add a \emph{function} (on $\mathcal{B}A$) to $\mathcal{B}A$, however, then this principle will break down. It is clear that the argument above will not work, because adding a function involves changing the application. In the next section, we will give an explicit counterexample (\cref{ex:counterexample}).
\end{ex}

Another example of \cref{thm:extension_BA} is given by the following result. Again, we stress that such a result does not hold if $\mathcal{B}A$ is taken to be an absolute PCA.
\begin{prop}\label{prop:surjection}
The applicative morphism $i\colon A\to \mathcal{B}A$ has a right adjoint $h\colon\mathcal{B}A\pf A$ satisfying $hi\simeq\id_A$. In particular, there is a geometric surjection $\rt(\mathcal{B}A)\epi\rt(A)$.
\end{prop}
\begin{proof}
By \cref{thm:extension_BA}, there is a (largest) partial applicative morphism $h\colon \mathcal{B}A\to A$ such that $hi\simeq\id_A$. We also have $ihi \simeq i$, and by \cref{ex:id_is_max}, this implies $ih\leq \id_{\mathcal{B}A}$.
\end{proof}

As a corollary, we see that, while $i$ makes every element of $\mathcal{B}A$ representable, it adds no new \emph{effectively} representable functions.
\begin{cor}
An element of $\mathcal{B}A$ is effectively representable w.r.t.\@ $i$ iff it is effectively representable in $A$ itself.
\end{cor}
\begin{proof}
This is immediate from the existence of a left inverse for $i$ in $\sf{pPCA}$, along with the fact that \emph{effective} representability transfers along partial applicative morphisms.
\end{proof}

We close this section with a novel topos-theoretic interpretation of the construction from the previous section, which freely adjoins a partial function to a PCA. For $f\in\mathcal{B}A$, we can adjoin $f$ as a function to $A$, yielding $\iota_f\colon A\to A[f]$, but we can also adjoin it as an \emph{element} to $\mathcal{B}A$, which gives $\iota_f\colon \mathcal{B}A\to\mathcal{B}A[f]$. In the proof of \cref{prop:everything_rep}, we constructed a $\rho\in(\mathcal{B}A)^\#$ such that $\rho\alpha$ is defined and represents $\alpha$ w.r.t.\@ $i$, for all $\alpha\in \mathcal{B}A$. From this, we easily deduce that $\rho f\in (\mathcal{B}A[f])^\#$ represents $f$ w.r.t.\@ $\iota_fi$. This means we get a factorisation:
\[\begin{tikzcd}
A \arrow[r, "i"] \arrow[d, "\iota_f"'] & \mathcal{B}A \arrow[d, "\iota_f"] \\
{A[f]} \arrow[r, "j", dashed]          & {\mathcal{B}A[f]}                
\end{tikzcd}\]
where $j$ acts as $i$ does, i.e., $j(a) = \downset\{\hat{a}\}$.

According to \cref{thm:extension_BA},
\[
h(\alpha) = \{a\in A\mid a\mbox{ represents }\alpha\mbox{ w.r.t.\@ }\iota_f\colon A\to A[f]\}
\]
defines a partial applicative morphism $h\colon \mathcal{B}A\pf A[f]$ such that $hi\simeq \iota_f$. By construction, $f$ is effectively representable w.r.t.\@ $\iota_f$, so $f\in \dom^\# h$. This means that $h$ factors through $\iota_f\colon\mathcal{B}A \to\mathcal{B}A[f]$ as $h\simeq k\iota_f$ for some $k\colon \mathcal{B}A[f]\pf A[f]$.
\[\begin{tikzcd}[row sep=large, column sep=large]
A \arrow[r, "i"] \arrow[d, "\iota_f"'] & \mathcal{B}A \arrow[d, "\iota_f"] \arrow[ld, "h"', dashed] \\
{A[f]} \arrow[r, "j"', shift right]    & {\mathcal{B}A[f]} \arrow[l, "k"', dashed, shift right]    
\end{tikzcd}\]
Now we have
\[
kj\iota_f \simeq k\iota_f i \simeq hi \simeq \iota_f,
\]
so by \cref{cor:universal_Af}, we have $kj\simeq \id_{A[f]}$. On the other hand,
\[
jk\iota_fi \simeq jhi \simeq j\iota_f \simeq \iota_f i.
\]
By \cref{ex:id_is_max}, this implies $jk\iota_f\leq \iota_f$, which yields $jk\leq \id_{\mathcal{B}A[f]}$. We conclude that $j\dashv k$ with $kj\simeq \id_{A[f]}$.

Recall from \cref{ex:iota_r_cd} that $\rt(\mathcal{B}A[f])$ is equivalent to the slice of $\rt(\mathcal{B}A)$ over the subterminal assembly $1_f$, which is given by $|1_f| = \{\ast\}$ and $E_{1_f}(\ast) = \downset\{f\}$. So at the level of toposes, we get the following diagram:
\[\begin{tikzcd}
\rt(\mathcal{B}A)/1_f \arrow[r, hook] \arrow[d, two heads] & \rt(\mathcal{B}A) \arrow[d, two heads] \\
{\rt(A[f])} \arrow[r, hook]                                & \rt(A)                                  
\end{tikzcd}\]
Intuitively, this diagram may be explained as follows. First, we cover $\rt(A)$ by the topos $\rt(\mathcal{B}A)$, where the truth values are sets of \emph{functions} rather than sets of elements of $A$. Then, we make $f$ computable (or alternatively: true) at \emph{that} level by taking the slice subtopos $\rt(\mathcal{B}A)/1_f$. Finally, $\rt(A[f])$ is retrieved as the geometric surjection-inclusion factorisation of the composition $\rt(\mathcal{B}A)/1_f \mono \rt(\mathcal{B}A) \epi \rt(A)$.

Another application of the construction of $j$ and $k$ above is the following nice characteization of the elements in $\mathcal{B}A$ that are effectively representable w.r.t.\@ $\iota_f\colon A\to A[f]$. Since $j$ has a left inverse, we have that $\alpha\in\mathcal{B}A$ is effectively representable w.r.t.\@ $\iota_f$ iff $\alpha$ is effectively representable w.r.t.\@ $j\iota_f\simeq \iota_f i$. By \cref{ex:id_is_max}, we know that $\alpha\in\mathcal{B}A$ and representers of $\alpha$ w.r.t.\@ $\iota_fi$ can be effectively translated into one another, so $\alpha$ is representable w.r.t.\@ $\iota_fi$ iff $\alpha\in (\mathcal{B}A[f])^\# = \langle (\mathcal{B}A)^\#\cup\{f\}\rangle$. We can conclude that
\[
\{\alpha\in\mathcal{B}A\mid \alpha\mbox{ is effectively representable w.r.t.\@ }\iota_f\colon A\to A[f]\} = \langle (\mathcal{B}A)^\#\cup\{f\}\rangle.
\]
In other words, the set of all elements of $\mathcal{B}A$ that are `forced' to become effectively representable if $f$ is effectively representable, can simply be obtained by adding $f$ to the filter in the PCA $\mathcal{B}A$.


\section{Adjoining a type-2 functional}\label{sec:type2}

In the previous section, we have explained how to freely adjoin zeroth- and first-order functions to a PCA $A$. This prompts the question whether a similar construction is available for \emph{second-order} functions. The answer to this question is yes, as the paper \cite{type2} shows for the absolute discrete case. In this section, we generalize this construction to \emph{some} relative ordered PCAs; it will in fact turn out that the order needs to be nice in a specific sense for the construction to work. There are two reasons for generalizing the material from \cite{type2}. First of all, the construction ties in quite nicely with \cref{thm:extension_BA} above in a way that had not been observed before. Second, in the next section we will apply this construction to $\mathcal{B}A$ (which is always an ordered PCA), in order to see what one can obtain for the \emph{third}-order case.

Let us first introduce the restriction on PCAs we need in order for the construction to work.
\begin{defn}\label{defn:cc}
Let $A$ be a PCA.
\begin{itemize}
\item[(i)]	A \emph{chain} in $A$ is a non-empty subset $X\subseteq A$ that is totally ordered by $\leq$.
\item[(ii)]	$A$ is called \emph{chain-complete} if every chain $X\subseteq A$ has a greatest lower bound in $A$, which we will denote by $\bigwedge X$.
\item[(iii)]	If $A$ is chain-complete, then a partial applicative morphism $g\colon A\pf B$ is called \emph{chain-continuous} if, for every chain $X\subseteq A$, we have $g\left(\bigwedge X\right) = \bigcap_{a\in X} g(a)$.
\end{itemize}
\end{defn}

We stress some important aspects of this definition.
\begin{rem}\label{rem:cc}
\begin{itemize}
\item[(i)]	Since we will not usually be interested in the greatest lower bound of $\emptyset$ (i.e., the top element), chains are non-empty by definition.
\item[(ii)]	If $g\colon A\pf B$ is chain-continuous, then it must preserve the order on the nose. Indeed, if $a'\leq a$, then $X = \{a',a\}$ is a chain with $\bigwedge X = a'$. It follows that $g(a') = g(a')\cap g(a)$, i.e., $g(a')\subseteq g(a)$.
\item[(iii)]	Of course, \cref{defn:cc} works just as well if $A$ and $B$ are merely posets, and $g$ is any function $A\to DB$. But we will only use these notions for PCAs and partial applicative morphisms.
\end{itemize}
\end{rem}

\begin{ex}\label{ex:disc=cc}
Every discrete PCA $A$ is chain-complete, since the only chains are singletons. For the same reason, every partial applicative morphism $g\colon A\pf B$ is trivially chain-continuous. This shows that our construction will subsume the discrete case.
\end{ex}

\begin{ex}\label{ex:cc_proj}
Suppose that $A$ is chain-complete, and that $g\colon A\to B$ is total and single-valued, i.e., we can write $g(a) = \downset\{g_0(a)\}$ for a certain \emph{function} $A\to B$. Then $g$ is chain-continuous if and only if, for each chain $X\subseteq A$, the greatest lower bound of $g_0(X)\subseteq B$ exists and is equal to $g_0(\bigwedge X)$.

From this, it is easy to deduce the following. Suppose that $A$ and $B$ are chain-complete, and that $A\stackrel{g}{\to} B\stackrel{h}{\pf} C$ are chain-continuous with $g$ projective. Then $hg$ is chain-continuous as well. This does \emph{not} seem to be true if we do not require that $g$ is projective.
\end{ex}

The following two lemmata discuss the compatibility between the notions introduced above and the construction $\mathcal{B}A$.

\begin{lem}\label{lem:cc_BA}
Let $A$ be a chain-complete PCA. Then $\mathcal{B}A$ is also chain-complete, and $i\colon A\to \mathcal{B}A$ is chain-continuous.
\end{lem}
\begin{proof}
Let $\{\alpha_i\mid i\in I\}$ be a chain in $\mathcal{B}A$. We define its greatest lower bound $\alpha = \bigwedge\{\alpha_i\mid i\in I\}$ as follows. First of all, we set:
\[\dom\alpha = \bigcup_{i\in I} \dom\alpha_i = \{a\in A\mid \exists i\in I\s (\alpha_i(a))\denotes\}.\]
For $a\in \dom \alpha$, the set $X_a = \{\alpha_i(a)\mid i\in I\mbox{ and }a\in \dom\alpha_i\}$ is a chain in $A$. This means we can define $\alpha(a) = \bigwedge X_a$. We leave it to the reader to show that $\alpha\in\mathcal{B}A$, and that $\alpha$ is indeed the greatest lower bound of the $\alpha_i$.

For the final statement, we observe that, for a chain $X\subseteq A$, we have $\bigwedge_{a\in X} \hat{a} = \widehat{\bigwedge X}$, which suffices by \cref{ex:cc_proj}.
\end{proof}

\begin{lem}\label{lem:cc_gh}
Consider a decidable partial applicative morphism $g\colon A\pf B$, where $A$ is chain-complete. Let $h\colon \mathcal{B}A\pf B$ be the largest partial applicative morphism such that $hi\simeq g$ as in \cref{thm:extension_BA}. If $g$ is chain-continuous, then $h$ is chain-continuous as well.
\end{lem}
\begin{proof}
Suppose that $g$ is chain-continuous. We recall that, since $g$ preserves the order on the nose, $h$ does so as well. If $\{\alpha_i\mid i\in I\}$ is a chain in $\mathcal{B}A$ and $\alpha=\bigwedge_{i\in I}\alpha_i$, then this already implies that $h(\alpha)\subseteq \bigcap_{i\in I} h(\alpha_i)$. So it remains to show the converse inclusion, i.e., if $b\in B$ represents all the $\alpha_i$ w.r.t.\@ $g$, then $b$ also represents $\alpha$ w.r.t.\@ $g$.

So suppose that $b\in B$ represents $\alpha_i$ w.r.t.\@ $g$, for all $i\in I$, and consider an $a\in\dom\alpha$. Then $a\in\dom\alpha_i$ for some $i$, which implies that $b\cdot g(a)$ is defined. Moreover, for all $i\in I$ such that $a\in\dom\alpha_i$, we have $b\cdot g(a)\subseteq g(\alpha_i(a))$. This yields:
\[
b\cdot g(a) \subseteq \bigcap_{\substack{i\in I\\ a\in\dom\alpha_i}} g(\alpha_i(a)) = g\left(\bigwedge \{\alpha_i(a)\mid i\in I\mbox{ and } a\in\dom\alpha_i\}\right) = g(\alpha(a)),
\]
as desired.
\end{proof}

The main point of introducing chain-completeness is that we can perform `fixpoint constructions' in $\mathcal{B}A$, which is crucial for generalizing the contruction from \cite{type2}. This construction is described in the following proposition.

\begin{prop}\label{prop:fixpoint}
Let $A$ be a chain-complete PCA and let $F\in\mathcal{B}\mathcal{B}A$ be a \emph{total} function. Then $F$ has a largest fixpoint in $\mathcal{B}A$.
\end{prop}
\begin{rem}
It may seem strange that we construct a \emph{largest} fixpoint, since recursion theory is usually concerned with \emph{smallest} fixpoints. However, we must keep in mind that `largest' should be read w.r.t.\@ the ordering on $\mathcal{B}A$ as in \cref{defn:BA_as_poset}. If $A$ is discrete, then this is the \emph{reverse} subfunction relation, so what we identify as the largest fixpoint would usually indeed be called the smallest fixpoint.
\end{rem}
\begin{proof}[Proof of \cref{prop:fixpoint}]
A total function $F\in\mathcal{BB}A$ is simply an order-preserving function $\mathcal{B}A\to\mathcal{B}A$.
We define, recursively, an ordinal-indexed sequence of element $f_\gamma$ of $\mathcal{B}A$, as follows:
\begin{itemize}
\item	$f_0 = \emptyset$;
\item	$f_{\gamma+1} = F(f_\gamma)$;
\item	$f_\lambda = \bigwedge_{\kappa<\lambda} f_\kappa$ if $\lambda>0$ is a limit ordinal.
\end{itemize}
Using transfinite induction and the fact that $F$ is order-preserving, one may show that $f_\gamma\geq f_\delta$ for $\gamma\leq\delta$, and that the sequence is well-defined. By cardinality considerations, the sequence must stabalize at some point, i.e., there exists an ordinal $\zeta$ such that $F(f_\zeta) = f_\zeta$. Then $f_\zeta$ is a fixpoint of $f$. Moreover, if $f'\in\mathcal{B}A$ is an element satisfying $f'\leq F(f')$, then by transfinite induction, it easily follows that $f'\leq f_\gamma$ for every ordinal $\gamma$, and in particular, $f'\leq f_\zeta$. We conclude that $f_\zeta$ is the largest fixpoint of $F$.
\end{proof}

\begin{prop}\label{prop:fixpoint_rep}
Consider a decidable partial applicative morphism $g\colon A\pf B$, such that $A$ is chain-complete and $g$ is chain-continuous. Let $F\in\mathcal{BB}A$ be a total function, and let $h\colon \mathcal{B}A\pf B$ be the largest partial applicative morphism such that $hi\simeq g$. If $F$ is effectively representable w.r.t.\@ $h$, then the largest fixpoint of $F$ is effectively representable w.r.t.\@ $g$.
\end{prop}
\begin{proof}
Let $\sf{z}\in B^\#$ be the guarded fixpoint operator. We will show that, if $r\in B$ represents $F$ w.r.t.\@ $h$, then $\sf{z}r$ (which is always defined) represents the largest fixpoint of $F$ w.r.t.\@ $g$. Clearly, this implies the proposition, for we have $\sf{z}r\in B^\#$ if $r\in B^\#$.

So suppose that $r\in B$ represents $F$ w.r.t.\@ $h$. Define the sequence $f_\gamma$ as in the proof of \cref{prop:fixpoint}, so that the largest fixpoint of $F$ is $f_\zeta$ for some ordinal $\zeta$. We will show, using transfinite induction, that $\sf{z}r\in h(f_\gamma)$ for all ordinals $\gamma$. In particular, we will have $\sf{z}r\in h(f_\zeta)$, which means that $\sf{z}r$ represents $f_\zeta$ w.r.t.\@ $g$, as desired.

First of all, we have $h(f_0) = h(\emptyset) = B$, so the base case is trivial. Now suppose that $\sf{z}r\in h(f_\gamma)$ for a certain ordinal $\gamma$. Then $r(\sf{z}r)\in h(F(f_\gamma)) = h(f_{\gamma+1})$, i.e., $r(\sf{z}r)$ represents $f_{\gamma+1}$ w.r.t.\@ $g$. Now, since $\sf{z}rb\preceq r(\sf{z}r)b$ for all $b\in B$, it follows that $\sf{z}r$ also represents $f_{\gamma+1}$ w.r.t.\@ $g$, i.e., $\sf{z}r\in h(f_{\gamma+1})$. Finally, \cref{lem:cc_gh} tells us that $h$ is chain-continuous, from which the limit case immediatelt follows. This completes the induction.
\end{proof}

Now let us introduce the type-2 functionals that we want to adjoin to a PCA $A$.

\begin{defn}
Let $A$ be a PCA.
\begin{itemize}
\item[(i)]	The set $\mathcal{B}_2A$ is defined as the set of all partial functions $\mathcal{B}A\pf A$ such that $\alpha\leq \beta$ implies $F(\alpha)\preceq F(\beta)$ for all $\alpha,\beta\in\mathcal{B}A$.
\item[(ii)]	For $F,G\in\mathcal{B}_2A$, we say that $F\leq G$ if $F(\alpha)\preceq G(\alpha)$ for all $\alpha\in\mathcal{B}A$.
\end{itemize}
\end{defn}

\begin{rem}
Observe that, even for discrete PCAs $A$, the set $\mathcal{B}_2A$ does not consist of all partial functions $\mathcal{B}A\pf A$. Indeed, in the discrete case, we have that $F\in\mathcal{B}_2A$ if and only if the following holds: whenever $\alpha$ is a subfunction of $\beta$ and $F(\alpha)\denotes$, we have that $F(\beta)$ is defined as well and equal to $F(\alpha)$. One may view this as an `extensionality' requirement: if $F(\alpha)$ is defined, then this must be based solely on the values specified by $\alpha$. Note, however, that we do not require that $F(\alpha)\denotes$ implies that the value of $F(\alpha)$ is already determined by \emph{finitely many} values of $\alpha$. In other words, it is possible that $F$ `consults' infinitely many values of $\alpha$. The paper \cite{type2} shows how to adjoin partial functions $A^A\pf A$, which are always in $\mathcal{B}_2A$ (if $A$ is discrete).
\end{rem}

\begin{defn}
Let $A$ be a PCA and let $F\in\mathcal{B}_2A$.
\begin{itemize}
\item[(i)]	We say that $r\in A$ \emph{represents} $F$ if: whenever $\alpha\in\dom F$ and $a\in A$ represents $\alpha$, we have $ra\leq F(\alpha)$. We say that $F$ is representable (resp.\@ effectively representable) if $F$ is represented by some $r\in A$ (resp.\@ $r\in A^\#$).
\item[(ii)]	If $g\colon A\pf B$ is a partial applicative morphism, then we say that $s\in B$ \emph{represents} $F$ w.r.t.\@ $g$ if: whenever $\alpha\in \dom F$ and $b\in B$ represents $\alpha$ w.r.t.\@ $g$, we have $sb\in g(F(\alpha))$. We say that $F$ is representable (resp.\@ effectively representable) w.r.t.\@ $g$ if $F$ is represented w.r.t.\@ $g$ by some $s\in B$ (resp.\@ $s\in B^\#$).
\end{itemize}
\end{defn}

Again, the first item is a special case of the second item, by taking $g=\id_A$.

\begin{rem}
The effective representability of second-order functionals is \emph{not} transferable along partial applicative morphisms. Below, we present an instructive example of this phenenomenon. It is also the reason why \cref{thm:type2} below cannot be restated as a `universal property' in the same vein as \cref{cor:universal_Af}. In this sense, the representability of second-order functionals is categorically ill-behaved when compared to the first-order case.
\end{rem}

\begin{ex}\label{ex:no_transfer}
Consider Kleene's first model $\mathcal{K}_1$ and let $F\colon \mathbb{N}^\mathbb{N}\to\mathbb{N}$ be defined by:
\[
F(f) = \begin{cases}
0 &\mbox{if }f\mbox{ is recursive};\\
1 &\mbox{otherwise}.
\end{cases}
\]
Since the domain of $F$ consists only of total functions, which form a discrete subset of $\mathcal{BK}_1$, we automatically have $F\in\mathcal{B}_2\mathcal{K}_1$. Moreover, $F$ is effectively representable in $\mathcal{K}_1$. Indeed, the only representable total functions in $\mathcal{K}_1$ are, by definition, the recursive functions, and $F$ is constant on those. However, $F$ is not even representable w.r.t.\@ $i\colon \mathcal{K}_1\to\mathcal{BK}_1$, let alone effectively so. Indeed, the representability of $F$ w.r.t.\@ $i$ would imply that the function $\hat{F}\in\mathcal{BBK}_1$ given by:
\[
\hat{F}(f) = \widehat{F(f)} = \begin{cases}
\hat{0} &\mbox{if }f\mbox{ is recursive};\\
\hat{1} &\mbox{otherwise}.
\end{cases}
\]
is representable in $\mathcal{BK}_1$. For, given $f\in\mathcal{BK}_1$, we can first effectively find a representer of $f$ w.r.t.\@ $i$, and then use the representer for $F$ w.r.t.\@ $i$ to obtain $\widehat{F(f)} = \hat{F}(f)$. However, $\hat{F}$ cannot be representable in $\mathcal{BK}_1$. Indeed, suppose that it is represented by $\rho\in\mathcal{BK}_1$. Then
\[
F(f) = \hat{F}(f)(0) = \rho f(0)
\]
for all $f\colon\mathbb{N}\to\mathbb{N}$. In particular, $\rho\hat{0}(0) = 0$. This computation consults the oracle $\hat{0}$ only finitely many times, so there exists an $N$ such that $F(f) = \rho f(0) = 0$ for \emph{all} $f\colon\mathbb{N}\to\mathbb{N}$ with $f(n) = 0$ for all $n<N$. This is clearly a contradiction, since this includes non-recursive $f$. (More generally, the point here is that a representer $\rho$ can consult only finitely many values of an $f\colon\mathbb{N}\to\mathbb{N}$ to determine $\rho f(0)$, but this does not suffice to determine whether $f$ is recursive.)
\end{ex}

Now we state and prove the main result of this section.
\begin{thm} \textup{(}Cf.\@ \cite{type2}, Theorem 3.1.\textup{)}\label{thm:type2}
Let $A$ be a chain-complete PCA and let $F\in\mathcal{B}_2A$. Then there exists an $f\in\mathcal{B}A$ such that:
\begin{itemize}
\item[\textup{(}i\textup{)}]	$F$ is effectively representable w.r.t.\@ $\iota_f\colon A\to A[f]$;
\item[\textup{(}ii\textup{)}]	if $g\colon A\pf B$ is decidable and chain-continuous, and $F$ is effectively representable w.r.t.\@ $g$, then $g$ factors, up to isomorphism, through $\iota_f$.
\end{itemize}
\end{thm}
\begin{proof}
Define the total function $\tilde{F}\in\mathcal{BB}A$ by:
\[
\tilde{F}(\alpha)(a) \simeq F(\lambda a'.a\odot_\alpha a').
\]
We leave it to the reader to show that $\tilde{F}$ is well-defined and an element of $\mathcal{BB}A$. Let $f\in\mathcal{B}A$ be the largest fixpoint of $\tilde{F}$, whose existence was asserted in \cref{prop:fixpoint}.

(i). Suppose that $a\in A$ represents $\alpha\in\dom F$ with respect to $\iota_f$. Then we have that $\lambda a'. a\odot_f a' \leq \alpha$, which implies that $f(a) \simeq F(\lambda a'.a\odot_f a')$ is defined and $f(a)\leq F(\alpha)$. So any representer of $f$ w.r.t.\@ $\iota_f$ also represents $F$ w.r.t.\@ $\iota_f$. Since $f$ is, by construction, effectively representable w.r.t.\@ $\iota_f$, it follows that $F$ is effectively representable w.r.t.\@ $\iota_f$ as well.

(ii). Let $h\colon\mathcal{B}A\pf B$ be the largest partial applicative morphism such that $hi\simeq g$. In order to show that $g$ factors through $\iota_f$, it suffices to show that $f$ is effectively representable w.r.t.\@ $g$. According to \cref{prop:fixpoint_rep}, this follows if we show that $\tilde{F}$ is effectively representable w.r.t.\@ $h$.

Let $s\in B^\#$ be such that $s\cdot h(\alpha)\subseteq g(F(\alpha))$ for all $\alpha\in\dom F$. By employing the usual kind of fixpoint argument, we may construct an $r\in B^\#$ such that $r\cdot h(\alpha) \cdot g(a)\cdot g(a') \subseteq g(a\odot_\alpha a')$ whenever $a\odot_\alpha a'\denotes$. Now we claim that $t := \lambda^\ast xy. s(\lambda^\ast z.rxyz)\in B^\#$ represents $\tilde{F}$ w.r.t.\@ $h$.

In order to prove this, consider $\alpha\in\mathcal{B}A$ and $a\in A$ such that $\tilde{F}(\alpha)(a)$ is defined, and take $b\in h(\alpha)$ and $c\in g(a)$. Moreover, consider $a'\in A$ such that $a\odot_\alpha a'$ is defined, and $c'\in g(a')$. Then we have that $((\lambda^\ast z.rxyz)[b/x,c/y])\cdot c'\preceq rbcc'\in g(a\odot_\alpha a')$, which means that $(\lambda^\ast z.rxyz)[b/x,c/y]$ represents $\lambda a'. a\odot_\alpha a'$. This implies that
\[
tbc \preceq s \cdot ((\lambda^\ast z.rxyz)[b/x,c/y]) \in g(F(\lambda a'.a\odot_\alpha a')) = g\left(\tilde{F}(\alpha)(a)\right).
\]
We can conclude that $tb$ represents $\tilde{F}(\alpha)$ w.r.t.\@ $g$, in other words, that $tb\in h(\tilde{F}(\alpha))$, as desired.
\end{proof}

If $A$ is chain-complete and $F\in \mathcal{B}_2A$, then we will denote $A[f]$ and $\iota_f$ constructed above by $A[F]$ and $\iota_F$, respectively.

We close this section with the counterexample announced at the end of \cref{ex:id_is_max}. It shows that, once we add an oracle to $\mathcal{B}A$, elements of $\mathcal{B}A$ and their representers can start to behave very differently. This is a serious obstruction to studying higher-order computability on $A$ by means of $\mathcal{B}A$.

\begin{ex}\label{ex:counterexample}
As in \cref{ex:no_transfer}, we let $A$ be Kleene's first model $\mathcal{K}_1$. Consider the \emph{Kleene functional} $E\colon\mathbb{N}^\mathbb{N}\to \mathbb{N}$ defined by:
\[
E(f) = \begin{cases}
0 &\mbox{if }\forall n\in\mathbb{N}\s (f(n)=0);\\
1 &\mbox{if }\exists n\in\mathbb{N}\s (f(n)>0),
\end{cases}
\]
which is in $\mathcal{B}_2\mathcal{K}_1$ since its domain consists of total functions. Also, as in \cref{ex:no_transfer}, we define $\hat{E}\in\mathcal{BBK}_1$ by $\hat{E}(f) = \widehat{E(f)}$ for $f\colon\mathbb{N}\to\mathbb{N}$. We will now proceed to show the following things about the composition
\[
\mathcal{K}_1\stackrel{i}{\longrightarrow}\mathcal{BK}_1\stackrel{\iota_{\hat{E}}}{\longrightarrow} \mathcal{BK}_1[\hat{E}].
\]
\begin{enumerate}
\item	Every total function $\mathbb{N}\to\mathbb{N}$ which is representable w.r.t.\@ $\iota_{\hat{E}}\circ i$ is arithmetical.
\item	The composition $\iota_{\hat{E}}\circ i$ does \emph{not} factor through $\iota_E\colon \mathcal{K}_1\to\mathcal{K}_1[E]$.
\item	The second-order functional $E$ is \emph{not} effectively representable w.r.t. $\iota_{\hat{E}}\circ i$.
\item	The morphism $\iota_{\hat{E}}$ is \emph{not} the largest partial applicative morphism $h\colon \mathcal{BK}_1\pf\mathcal{BK}_1[\hat{E}]$ such that $hi\simeq \iota_{\hat{E}}\circ i$.
\end{enumerate}

For the first claim, we first observe the following: since the range of $\hat{E}$ is simply $\{\hat{0},\hat{1}\}\subseteq (\mathcal{BK}_1)^\#$, we have $\left(\mathcal{BK}_1[\hat{E}]\right)^\# = (\mathcal{BK}_1)^\#$, which is just the set of partial recursive functions. Now, if $\alpha,\beta,\gamma\in\mathcal{BK}_1$ are partial functions, then the relation `$\alpha\beta = \gamma$' can be expressed arithmetically in terms of the graphs of $\alpha$, $\beta$ and $\gamma$, as is immediate from the definition of application in $\mathcal{BK}_1$. Moreover, if $\alpha\in\mathcal{BK}_1$ and $i\in\mathbb{N}$, then $\hat{E}(\alpha) = \hat{i}$ can be expressed arithmetically in terms of $i$ and the graph of $\alpha$. And we know that all the combinators in $\mathcal{BK}_1$ are partial recursive functions, and therefore have arithmetically expressible graphs. Finally, we assume for simplicity that the booleans in $(\mathcal{BK}_1)^\#$ are simply $\hat{0}$ and $\hat{1}$, which we can, because $i\colon\mathcal{K}_1\to\mathcal{BK}_1$ is decidable. Now suppose that $f\colon \mathbb{N}\to\mathbb{N}$ is represented w.r.t.\@ $\iota_{\hat{E}}\circ i$ by $\rho\in \left(\mathcal{BK}_1[\hat{E}]\right)^\# = (\mathcal{BK}_1)^\#$. Then $\rho$ is partial recursive, and since $f$ is total, we have $f(a) = b$ iff $\rho\odot_{\hat{E}} \hat{a} = \hat{b}$. We see that $\rho\odot_{\hat{E}} \hat{a} = \hat{b}$ holds iff there exists a coded sequence $u = [u_0, \ldots, u_{n-1}]$ of \emph{natural numbers}, such that:
\begin{itemize}
\item	for all $i<n$, we have $\sf{p}_0(\rho\cdot [\hat{a}, \hat{u}_0, \ldots, \hat{u}_{i-1}]) = \hat{0}$ and $\hat{E}\left(\sf{p}_1(\rho\cdot [\hat{a}, \hat{u}_0, \ldots, \hat{u}_{i-1}])\right) = \hat{u}_i$, and:
\item	$\sf{p}_0(\rho\cdot [\hat{a}, \hat{u}_0, \ldots, \hat{u}_{n-1}]) = \hat{1}$ and $\sf{p}_1(\rho\cdot [\hat{a}, \hat{u}_0, \ldots, \hat{u}_{n-1}]) = \hat{b}$.
\end{itemize}
By the remarks above, this is an arithmetical relation in terms of $a$ and $b$, so we see that `$f(a)=b$' is arithmetical, as desired.

The paper \cite{type2} shows that, when \cref{thm:type2} is applied to Kleene's first model, the result is equivalent to Kleene's original notion of computability w.r.t.\@ a higher-order functional. In particular, the partial functions $\mathbb{N}\pf\mathbb{N}$ which are effectively representable w.r.t.\@ $\iota_E$ are precisely the \emph{hyper}arithmetical functions (\cite{type2}, Corollary 4.1). Now, if $\iota_{\hat{E}}\circ i$ were to factor through $\iota_E$, then every hyperarithmetical functions would be effectively representable w.r.t.\@ $\iota_{\hat{E}}\circ i$ as well. But claim 1 tells us that this is not the case, since there are certainly total hyperarithmetical functions which are not arithmetical (e.g., the characteristic function of a hyperarithemtical set which is not arithmetical).

Claim 3 immediately follows from claim 2 by \cref{thm:type2}.

Finally, in order to prove claim 4, suppose for the sake of contradiction that $\sigma \in (\mathcal{BK}_1)^\#$ realizes the inequality $h\leq\iota_{\hat{E}}$, where $h(\alpha) = \{\beta\in\mathcal{BK}_1\mid\beta\mbox{ represents }\alpha\mbox{ w.r.t.\@ }\iota_{\hat{E}}\circ i\}$. Moreover, let $\rho\in (\mathcal{BK}_1)^\#$ represent $\hat{E}$ w.r.t.\@ $\iota_{\hat{E}}$. Then it easily follows that
\[
\lambda^\ast x. \rho \odot_{\hat{E}} (\sigma\odot_{\hat{E}} x)\in (\mathcal{BK}_1)^\#
\]
represents $E$ w.r.t.\@ $\iota_{\hat{E}}\circ i$, which is not the case by claim 3.
\end{ex}


\section{The third-order case}\label{sec:type3}

In this final section, we investigate what can be achieved in the case of \emph{third}-order functionals. First, let us introduce the objects considered in this section.

\begin{defn}\label{defn:B3A}
Let $A$ be a PCA.
\begin{itemize}
\item[(i)]	The set $\mathcal{B}_3A$ consists of all partial functions $\Phi\colon \mathcal{B}_2A\pf A$ such that $\Phi(F)\preceq \Phi(G)$ whenever $F\leq G$.
\item[(ii)]	We say that $r\in A$ \emph{represents} $\Phi\in\mathcal{B}_3A$ if: whenever $F\in\dom \Phi$ and $a\in A$ represents $F$, we have $ra\leq \Phi(F)$. We say that $\Phi$ is representable (resp.\@ effectively representable) if $\Phi$ is represented by some $r\in A$ (resp.\@ $r\in A^\#$).
\item[(iii)]	If $g\colon A\pf B$ is a partial applicative morphism, then we say that $s\in B$ \emph{represents} $\Phi$ w.r.t.\@ $g$ if: whenever $F\in \dom \Phi$ and $b\in B$ represents $F$ w.r.t.\@ $g$, we have $sb\in g(\Phi(F))$. We say that $\Phi$ is representable (resp.\@ effectively representable) w.r.t.\@ $g$ if $\Phi$ is represented w.r.t.\@ $g$ by some $s\in B$ (resp.\@ $s\in B^\#$).
\end{itemize}
\end{defn}

As the authors of \cite{type2} mention as well, there is a fundamental obstacle when studying the representability of third-order functionals. A PCA $A$ can only `talk about' first-order functions by means of their representers. So as far as $A$ is concerned, only representable functions really exist. For representing second-order functionals $F$, this is not a problem. On the contrary, is makes the task easier: if a first-order function $\alpha$ does not have a representer, then we do not have to worry about $\alpha$ when constructing a representer for $F$. In other words, from the point of view of $A$, the domain of $F$ may look smaller than it actually is.

For a third-order functional $\Phi$, on the other hand, the second-order functionals serve as \emph{inputs}, and $A$ may lose information contained in these inputs. More precisely, there may be two distinct $F,G\in \mathcal{B}_2A$ such that $\Phi(F)$ and $\Phi(G)$ do not have a common lower bound, but which are the same from the point of view of $A$. Let us give an example of this phenemenon.

\begin{ex}\label{ex:3rd_order_bad}
Again, we let $A$ be Kleene's first model $\mathcal{K}_1$. Consider the third-order functional $\Phi\in\mathcal{B}_3\mathcal{K}_1$ defined by $\dom\Phi = \{F\in\mathcal{B}_2\mathcal{K}_1\mid \mathbb{N}^\mathbb{N}\subseteq \dom F\}$ and
\[
\Phi(F) = \begin{cases}
0 &\mbox{if }\forall f\in\mathbb{N}^\mathbb{N}\s (F(f)=0);\\
1 &\mbox{if }\exists f\in\mathbb{N}^\mathbb{N}\s (F(f)>0).
\end{cases}
\]
We leave it to the reader to check that $\Phi$ is actually in $\mathcal{B}_3\mathcal{K}_1$. Recall from \cref{ex:no_transfer} the functional $F\colon \mathbb{N}^\mathbb{N}\to\mathbb{N}$ defined by
\[
F(f) = \begin{cases}
0 &\mbox{if }f\mbox{ is recursive};\\
1 &\mbox{otherwise}.
\end{cases}
\]
Consider also the function $G\in\mathcal{B}_2\mathcal{K}_1$, which is 0 on total functions, and undefined on non-total functions. Then we clearly have $\Phi(F) = 1 \neq 0 = \Phi(G)$. However, from the point of view of $\mathcal{K}_1$, the functionals $F$ and $G$ are equal, and both are represented by an index for the constant 0 function.

This does not yet exclude the possibility that, as for the second-order case, we can construct a partial function $g\colon\mathbb{N}\pf\mathbb{N}$ such that $\Phi$ becomes representable w.r.t.\@ $\iota_g\colon \mathcal{K}_1\to\mathcal{K}_1[g]$. However, we can adjust the example above to show that this is impossible as well. Define $F_g\colon\mathbb{N}^\mathbb{N}\to\mathbb{N}$ by:
\[
F_g(f) = \begin{cases}
0 &\mbox{if }f\mbox{ is representable w.r.t.\@ } \iota_g;\\
1 &\mbox{otherwise}.
\end{cases}
\]
Once again, we have $\Phi(F_g) = 1 \neq 0 = \Phi(G)$, but $A[g]$ cannot distinguish between $F_g$ and $G$. This means that $\Phi$ is not representable (let alone \emph{effectively} representable) w.r.t.\@ to $\iota_g$ for \emph{any} partial function $g$.
\end{ex}

This example shows that a construction as in the previous section, where $A[F]$ was of the form $A[f]$ for an $f\in\mathcal{B}A$, simply cannot work in the third-order case. On the other hand, the example above would clearly be blocked if we move to a PCA in which every (partial) function is representable. Fortunately, we have such a PCA, namely $\mathcal{B}A$, which allows us to prove the following `lax' result about the third-order case.

\begin{thm}\label{thm:type3}
Let $A$ be a chain-complete PCA and let $\Phi\in\mathcal{B}_3A$. Then there exists a c.d.\@ total applicative morphism $\iota_\Phi\colon A\to A[\Phi]$ such that:
\begin{itemize}
\item[\textup{(}i\textup{)}]	$\Phi$ is effectively representable w.r.t.\@ $\iota_\Phi$;
\item[\textup{(}ii\textup{)}]	if $g\colon A\pf B$ is decidable and chain-continuous, and $\Phi$ is effectively representable w.r.t.\@ $g$, then there exists a \emph{largest} $h\colon A[\Phi]\pf B$ such that $h\iota_\Phi \simeq g$.
\end{itemize}
\end{thm}
\begin{proof}
Define $\tilde{\Phi}\in \mathcal{B}_2\mathcal{B}A$ by
\[
\tilde{\Phi}(F)(a) \simeq \Phi(\lambda \alpha. F(\alpha)(\sf{i})) \quad\mbox{for }F\in\mathcal{BB}A\mbox{ and }a\in A.
\]
Observe that if $\tilde{\Phi}(F)$ is defined, then it is a constant function. We leave it to the reader to check that $\tilde{\Phi}$ is actually in $\mathcal{B}_2\mathcal{B}A$. As before, if $F\in \mathcal{B}_2A$, then we define $\hat{F}\in\mathcal{BB}A$ by $\hat{F}(\alpha) \simeq \widehat{F(\alpha)}$. Since 
\[
\tilde{\Phi}(\hat{F})(a) \simeq \Phi(\lambda \alpha. \hat{F}(\alpha)(\sf{i})) \simeq \Phi(\lambda\alpha.F(\alpha)) \simeq \Phi(F),
\]
we have $\tilde{\Phi}(\hat{F}) \simeq \widehat{\Phi(F)}$ for all $F\in\mathcal{B}_2A$.

We will show that $\iota_\Phi\colon A\to A[\Phi]$ can be taken to be the composition:
\[
\begin{tikzcd}[column sep=large]
A \arrow[r, "i"] & \mathcal{B}A \arrow[r, "\iota_{\tilde{\Phi}}"] & {\mathcal{B}A[\tilde{\Phi}]}
\end{tikzcd}
\]
For the sake of readbility, we will just write $\iota$ for $\iota_{\tilde{\Phi}}$, and we write $\odot$ for the application in $\mathcal{B}A[\tilde{\Phi}]$.

\textbf{(i)} By construction, $\tilde{\Phi}$ is representable w.r.t.\@ $\iota$ by means of a $\rho\in(\mathcal{B}A[\tilde{\Phi}])^\#$. By \cref{thm:extension_BA}, we also know that there exists a $\sigma\in (\mathcal{B}A[\tilde{\Phi}])^\#$ such that $\sigma\odot \alpha$ represents $\alpha$ w.r.t.\@ $\iota\circ i$ for all $\alpha\in\mathcal{B}A$. We will show that
\[
\tau = \lambda^\ast x.\rho\odot (\lambda^\ast y.x\odot(\sigma\odot y))\in (\mathcal{B}A[\tilde{\Phi}])^\#
\]
represents $\Phi$ w.r.t.\@ $\iota\circ i$. So let $F\in\mathcal{B}_2A$ be such that $\Phi(F)$ is defined, and suppose that $\beta\in \mathcal{B}A$ represents $F$ w.r.t.\@ $\iota\circ i$. First of all, we claim that $(\lambda^\ast y.x\odot(\sigma\odot y))[\beta/x]$ represents $\hat{F}$ w.r.t.\@ $\iota$. In order to show this, let $\alpha\in\mathcal{B}A$ be such that $\hat{F}(\alpha)$ is defined. Then:
\[
(\lambda^\ast y.x\odot(\sigma\odot y))[\beta/x] \odot \alpha \preceq \beta\odot (\sigma\odot \alpha) \preceq \widehat{F(\alpha)} = \hat{F}(\alpha),
\]
since $\sigma\odot\alpha$ represents $\alpha$ w.r.t.\@ $\iota\circ i$ and $\beta$ represents $F$ w.r.t.\@ $\iota\circ i$. This proves the claim, and it follows that
\[
\tau\odot\beta \preceq \rho\odot \left((\lambda^\ast y.x\odot(\sigma\odot y))[\beta/x]\right) \preceq \tilde{\Phi}(\hat{F}) = \widehat{\Phi(F)},
\]
as desired.

\textbf{(ii)} Suppose that $s\in B^\#$ represents $\Phi$ w.r.t.\@ $g$, and consider $h\colon \mathcal{B}A\pf B$ defined by $h(\alpha) = \{b\in B\mid b\mbox{ represents }\alpha\mbox{ w.r.t.\@ }g\}$. Since $h$ is also decidable and chain-continuous, it suffices to show that $\tilde{\Phi}\in \mathcal{B}_2\mathcal{B}A$ is representable w.r.t.\@ $h$. Then \cref{thm:type2} tells us that $h$ is also a morphism $\mathcal{B}A[\tilde{\Phi}]\pf B$, and \cref{thm:extension_BA} implies that this is the largest partial applicative morphism by means of which $g$ factors through $\iota\circ i$.

We will show that
\[
t = \lambda^\ast x.\sf{k}(s(\lambda^\ast y.xyj))\in B^\#
\]
represents $\tilde{\Phi}$ w.r.t.\@ $h$, where $j$ is any element from $g(\sf{i})\cap B^\#$. So let $F\in\mathcal{BB}A$ be such that $\tilde{\Phi}(F)$ is defined, and suppose that $b\in B$ represents $F$ w.r.t.\@ $h$. First of all, we claim that $(\lambda^\ast y.xyj)[b/x]$ represents $\lambda\alpha. F(\alpha)(\sf{i})$ as a second-order functional w.r.t.\@ $g$. In order to prove the claim, let $\alpha\in\mathcal{B}A$ be such that $F(\alpha)(\sf{i})$ is defined, and let $c\in B$ represent $\alpha$ w.r.t.\@ $g$. Then $c\in h(\alpha)$, which means that $bc$ is defined and in $h(F(\alpha))$, i.e., $bc$ represents $F(\alpha)$ w.r.t.\@ $g$. Since $j\in g(\sf{i})$, this yields that $\left((\lambda^\ast y.xyj)[b/x]\right)c\preceq bcj$ is defined and an element of $g(F(\alpha)(\sf{i}))$, which proves the claim. Now we find that $s\left((\lambda^\ast y.xyj)[b/x]\right)\in\Phi(\lambda\alpha.F(\alpha)(\sf{i}))$, so it follows that $tb\preceq \sf{k}\left(s\left((\lambda^\ast y.xyj)[b/x]\right)\right)$ is defined and represents $\tilde{\Phi}(F)$. In other words, we have $tb\in h(\tilde{\Phi}(F))$, as desired.
\end{proof}

In \cref{sec:type2}, we investigated the possibility of making an $F\in\mathcal{B}_2A$ effectively representable by adjoining $\hat{F}\in\mathcal{BB}A$ to $\mathcal{B}A$. It turned out that this will not work in general, the Kleene functional $E\in\mathcal{B}_2\mathcal{K}_1$ being a counterexample. Therefore, it may seem strange that a similar strategy \emph{does} work for the third-order case! Let us explain why this is so. In the second-order case, the task was to construct, given a representer of $\hat{F}$, a representer of $F$. Now, a representer of $F$ eats representers of $\alpha\in\mathcal{B}A$, whereas a representer of $\hat{F}$ wants to eat $\alpha$ itself. So the task really is to effectively find, given a representer of $\alpha$, the function $\alpha$ itself so that it can be fed to the representer of $\hat{F}$. But the problem is exactly that, once we add an \emph{oracle} to $\mathcal{B}A$, this is no longer possible in general. On the other hand, the \emph{converse} construction obviously does work, i.e., given a representer of $F$, we can construct a representer of $\hat{F}$ (see also \cref{ex:no_transfer}). Since we constructed $\tilde{\Phi}$ in such a way that $\tilde{\Phi}(\hat{F})\simeq \widehat{\Phi(F)}$, this is precisely what we need to construct a representer for $\Phi$, given a representer for $\tilde{\Phi}$. We can also put this as follows: in the business of representing $\Phi$, the representers of second-order functionals $F\in \mathcal{B}_2A$ are not the things to be constructed, but the things that are \emph{given}. Unfortunately, this also reveals that the current strategy can probably not be pushed beyond the third-order case, because in the fourth-order case, things will be `the wrong way around' again.

We close the paper with an interesting corollary of \cref{thm:type3}. Suppose we have a decidable partial applicative morphism $g\colon A\pf B$ and an $f\in\mathcal{B}A$ which is effectively representable w.r.t.\@ $g$. Then we know that, besides $f$ and the functions that were already effectively representable in $A$, other functions must become effectively representable w.r.t.\@ $g$ as well. So a natural question to ask here is: what is the minimal set of elements of $\mathcal{B}A$ that must become effectively representable w.r.t.\@ a partial applicative morphism as soon as $f$ is effectively representable? We can reformulate the question as follows. Let $\er(f)$ denote the class of all decidable partial applicative morphism $g\colon A\pf B$ with respect to which $f$ is effectively representable. Then we are interested in the set:
\[
\foe(f):= \bigcap_{g\in\er(f)}\{\alpha\in\mathcal{B}A\mid\alpha\mbox{ is effectively representable w.r.t.\@ }g\},
\]
where $\foe$ should be read as `first-order effect'. The construction from \cref{sec:order1} tells us what this set is: indeed, we know that $\iota_f\colon A\to A[f]$ is in $\er(f)$ and that every element of $\er(f)$ factors through $\iota_f$. This means that $\foe(f)$ is simply the set of $\alpha\in\mathcal{B}A$ that are effectively representable w.r.t.\@ $\iota_f$, which we know to be $\langle (\mathcal{B}A)^\#\cup\{f\}\rangle$.

As we know, the representability of second-order functionals does not behave as nicely from the point of view of $\sf{pPCA}$. But we can still ask, given an $F\in\mathcal{B}_2A$, which \emph{first-order} functions are forced to become effectively representable if $F$ is effectively representable. For simplicity, we will assume that $A$ is discrete, so that we do not have to worry about chain-completeness and chain-continuity. As above, we let $\er(F)$ denote the class of all decidable partial applicative morphisms $g\colon A\pf B$ such that $F$ is effectively representable w.r.t.\@ $g$, and we write
\[
\foe(F):= \bigcap_{g\in\er(F)}\{\alpha\in\mathcal{B}A\mid\alpha\mbox{ is effectively representable w.r.t.\@ }g\},
\]
for the `first-order effect' of $F$. Then $\foe(F)$ is also of the form $\langle (\mathcal{B}A)^\#\cup\{f\}\rangle$ for some $f\in\mathcal{B}A$, namely, the $f$ as in \cref{thm:type2}.

For the third-order case, we can pose the analogous question. Let $A$ be a discrete PCA and consider $\Phi\in\mathcal{B}_3A$. Let $\er(\Phi)$ denote the class of all decidable partial applicative morphisms $g\colon A\pf B$ such that $\Phi$ is effectively representable w.r.t.\@ $g$, and we write
\[
\foe(\Phi):= \bigcap_{g\in\er(\Phi)}\{\alpha\in\mathcal{B}A\mid\alpha\mbox{ is effectively representable w.r.t.\@ }g\},
\]
for the `first-order effect' of $\Phi$. Then \cref{thm:type3} tells us that $\foe(\Phi)$ consists of all $\alpha\in\mathcal{B}A$ that are effectively representable w.r.t.\@ $\iota_{\Phi}$. We can now ask: is this set also of the form $\langle (\mathcal{B}A)^\#\cup\{f\}\rangle$ for some $f\in\mathcal{B}A$? In order to show that the answer is `yes', we first need the following definition.

\begin{defn}
Let $A$ be a discrete PCA and let $g\colon A\pf B$ be a partial applicative morphism. Then $g$ is called \emph{discrete} if $g(a)\cap g(a')=\emptyset$ for every two \emph{distinct} $a,a'\in A$.
\end{defn}

The following result is based on Theorem 2.12 of \cite{faberjaap}.
\begin{lem}\label{lem:insur}
Let $A$ be a discrete PCA and let $g\colon A\to B$ be a \emph{total} applicative morphism which is discrete, projective and c.d. Then there exists an $f\in \mathcal{B}A$ such that:
\[
\{\alpha\in\mathcal{B}A\mid \alpha\mbox{ is effectively representable w.r.t.\@ }g\} = \langle(\mathcal{B}A)^\#\cup\{f\}\rangle.
\]
\end{lem}
\begin{proof}
Since $g$ is total, projective and c.d., \cref{thm:cd=rightadjoint}(i) tells us that $g$ has a right adjoint $h\colon B\pf A$ in $\sf{pPCA}$. We claim that the partial applicative morphism $hg\colon A\pf A$ is discrete. Suppose we have $a,a'\in A$ such that $hg(a)\cap hg(a')$ is nonempty. Since $g$ preserves the order up to a realizer, it follows that $ghg(a)\cap ghg(a')$ is also nonempty. Since $g\dashv h$, we have $ghg\simeq g$, so this implies that $g(a)\cap g(a')$ is nonempty, so $a=a'$ by the discreteness of $g$.

Now define the partial function $f\colon A\pf A$ by: $f(a) =a'$ if and only if $a\in hg(a')$, which is well-defined by the discreteness of $hg$. Then $g(a) \subseteq ghg(f(a))$, so if $s\in B^\#$ realizes $ghg\leq g$, then $s$ also represents $f$ w.r.t.\@ $g$. This means that $g$ factors through $\iota_f$ by means of a $g'\colon A[f]\to B$, which is defined simply by $g'(a) = g(a)$ for $a\in A$. Now $g'$ is total, projective and c.d.\@ as well, so it has a right adjoint $h'\colon B\pf A$. Moreover, we recall that $\iota_f$ has a right adjoint $k\colon A[f]\to A$ satisfying $\iota_f k\simeq\id_{A[f]}$.
\[
\begin{tikzcd}[column sep=huge, row sep=huge]
A \arrow[rd, "\iota_f", shift left] \arrow[rr, "g", shift left] &                                                                 & B \arrow[ll, "h", shift left] \arrow[ld, "h'", shift left] \\
                                                                & {A[f]} \arrow[ru, "g'", shift left] \arrow[lu, "k", shift left] &                                                           
\end{tikzcd}
\]
Since $g\simeq g'\iota_f$, we also have $h\simeq kh'$, hence $\iota_f h\simeq \iota_fkh' \simeq h'$. This means we can assume without loss of generality that $h'(b) = h(b)$ for all $b\in B$. In particular, $h'g'(a) = hg(a)$ for all $a\in A$. But now it is clear that any representer $r\in (A[f])^\#$ of $f$ will also realize the inquality $h'g'\leq \id_{A[f]}$. Combining this with $g'\dashv h'$ yields $g'h'\simeq \id_{A[f]}$. This implies that $\alpha\in\mathcal{B}A$ is effectively representable w.r.t.\@ $g\simeq g'\iota_f$ if and only if $\alpha$ is effectively representable w.r.t.\@ $\iota_f$; if and only if $\alpha\in\langle(\mathcal{B}A)^\#\cup\{f\}\rangle$.
\end{proof}

\begin{rem}
Observe that the proof of \cref{lem:insur} also implies that the image topos of the geometric morphism $\rt(B)\to \rt(A)$ induced by $g$ is also a realizability topos, namely $\rt(A[f])$.
\end{rem}

\begin{cor}
Let $A$ be a discrete PCA and let $\Phi\in\mathcal{B}_3A$. Then there exists an $f\in\mathcal{B}A$ such that $\foe(\Phi) = \langle(\mathcal{B}A)^\#\cup\{f\}\rangle$.
\end{cor}
\begin{proof}
This follows from \cref{lem:insur} if we can show that $\iota_\Phi$ is total, discrete, projective and c.d. All of these are easy to check.
\end{proof}

\end{document}